\documentclass[12pt]{amsart}

\usepackage{amssymb,latexsym}
\usepackage{color}
\usepackage{comment}
\usepackage{enumerate}
\usepackage{graphicx}
\usepackage{bbm, dsfont}
\makeatletter
\usepackage{caption}
\@namedef{subjclassname@2010}{
  \textup{2010} Mathematics Subject Classification}
\makeatother
\newtheorem{thm}{Theorem}[section]

\newtheorem{Con}[thm]{Conjecture}
\newtheorem{Cor}[thm]{Corollary}
\newtheorem{Pro}[thm]{Proposition}
\newtheorem{cor}[thm]{Corollary}
\newtheorem{lem}[thm]{Lemma}
\newtheorem{pro}[thm]{Proposition}
\theoremstyle{definition}

\newtheorem{rem}[thm]{Remark}

\numberwithin{equation}{section}
\newcommand{\T}{\mathbf{t}}
\newcommand{\X}{\mathbb{X}}
\newcommand{\U}{\mathbf{u}}
\newcommand{\V}{\mathbf{v}}
\newcommand{\s}{\mathbf{s}}
\newcommand{\pr}{\mathbb{P}}
\newcommand{\ex}{\mathbb{E}}

\newcommand{\F}{\mathcal{F}(x)}

\newcommand{\G}{\mathcal{G}}
\newcommand{\J}{\mathcal{J}}
\newcommand{\ep}{\varepsilon}

\begin{document}

\baselineskip=16pt

\title[Sign changes of character sums and zeros of Fekete polynomials]{Sign changes of short character sums and real zeros of Fekete polynomials}

\author{Oleksiy Klurman}
\address{School of Mathematics, University of Bristol, UK}

\email{oleksiy.klurman@bristol.ac.uk}

\author{Youness Lamzouri}

\address{
Universit\'e de Lorraine, CNRS, IECL,  and  Institut Universitaire de France,
F-54000 Nancy, France}

\email{youness.lamzouri@univ-lorraine.fr}

\author{Marc Munsch}
\address{Université Jean Monnet, Institut Camille Jordan, 42 100 Saint-Etienne, France. }
\email{marc.munsch@univ-st-etienne.fr}

\date{}

\subjclass[2010]{Primary 11L40, 11M06, 26C10,  30C15; Secondary 11M20.}

\keywords{Dirichlet characters, $L$-functions, random model, discrepancy, multiplicative functions, Littlewood polynomials.}

\maketitle

\begin{abstract}
We discuss a general approach producing quantitative bounds on the number of sign changes of the weighted sums 
\[\sum_{n\le x}f(n)w_n\]
where $f:\mathbb{N}\to \mathbb{R}$ is a family of multiplicative functions and $w_n\in\mathbb{R}$ are certain weights.\\
As a consequence, we show that for a typical fundamental discriminant $D,$ the partial sums of the real character $\chi_D$ change sign $\gg (\log\log D)/\log\log\log \log D$ times on a {\it very short} initial interval (which goes beyond the range in Vinogradov's conjecture).\\ 
We also prove that the number of real zeros (localized away from $1$) of the Fekete polynomial associated to a typical fundamental discriminant $D$ is $\gg \frac{\log\log D}{\log\log\log\log D}.$ This comes close to establishing a conjecture of Baker and Montgomery which predicts $\asymp \log \log  D$ real zeros.\\ 
Finally, the same approach shows that almost surely for large $x\ge 1$, the partial sums 
$\sum_{n\le y}f(n)$ of a (Rademacher) random  multiplicative function exhibit $\gg \log \log x/\log \log\log\log x$ sign changes on the interval $[1,x].$\\
These results rely crucially on uniform quantitative estimates for the joint distribution of $-\frac{L'}{L}(s, \chi_D)$ at several points $s$ in the vicinity of the central point $s=1/2$, as well as concentration results for $\log L(s,\chi_D)$ in the same range, which we establish.\\
In the second part of the paper, we obtain, for large families of discriminants, ``non-trivial" upper bounds on the number of real zeros of Fekete polynomials, breaking the square root bound. Finally, we construct families of discriminants with associated Fekete polynomials having no zeros away from $1.$

\end{abstract}

\tableofcontents
\section{Introduction}
The main objective of the present paper is to illustrate a general approach producing quantitative bounds on the number of sign changes of the weighted sums 
\[\sum_{n\le x}f(n)w_n\]
where $f:\mathbb{N}\to \mathbb{R}$ is a family of multiplicative functions and $w_n\in\mathbb{R}$ are certain weights. In what follows, we shall focus on the case $f(n)=\chi_D(n)$ is a real Dirichlet character.\\ 
Our first application is related to a relatively ``thin'' family of polynomials, namely the so-called Fekete polynomials.
To this end, we recall that the Fekete polynomial of degree $\vert D\vert-1$ is defined by
$$F_D(z):=\sum_{n=1}^{\vert D\vert-1} \chi_D(n)z^n,$$ where $D$ is any fundamental discriminant\footnote{Recall that $D$ is a fundamental discriminant if $D$ is square-free and $D\equiv 1\pmod 4$, or $D=4m$ where $m\equiv 2\text{ or } 3 \pmod 4$ and $m$ is square-free.} and $\chi_D$ the associated real primitive Dirichlet character.
Erd\H{o}s and Littlewood \cite{littlewood1968some} raised various extremal problems concerning polynomials with coefficients $\pm 1$. The family of Fekete polynomials appears frequently in this context  (see for example \cite{BC}, \cite{borwein2001extremal}, \cite{M-N-T-1}, \cite{FeketePolya}, \cite{Hoholdtmerit}, \cite{JKS}, \cite{jedwab2013littlewood}, \cite{JHH}, \cite{K-L-M}, \cite{M-N-T}, \cite{MR4490459}, \cite{Mont-large}) and has been extensively studied for over a century \cite{FeketePolya}.
Dirichlet was the first to discover the link between a Fekete polynomial and its associated Dirichlet $L$-function:
\begin{equation}\label{Dirichlet-identity}L(s,\chi_D)\Gamma(s)=\int_0^1 \frac{(-\log x)^{s-1}}{x}\frac{F_D(x)}{1-x^D}dx, \text{ for } \Re(s)>0.\end{equation}
Hence if $F_D$ has no zeros for $0<x<1,$ then $L(s,\chi_D)>0$ for all $0<s<1,$ which in turn refutes the existence of a putative Siegel zero (note that the positivity of $L(s,\chi_D)$ on $(0,1)$ is known only  for a positive proportion of discriminants by a result of Conrey and Soundararajan \cite{ConreySound}).  

Here we recall an amusing story: a quick computation shows that there is only a few small discriminants $D$ for which $F_D$ has a zero in $(0,1)$ (see \cite{Lounumerics} for a recent numerical study).  This led Fekete to conjecture that $F_D$ does not have  zeros in $(0,1)$ 
for sufficiently large $D.$ Such a statement has been shortly disproved by P\'olya~\cite{Pol}. Twenty years later, unaware of Fekete's conjecture and its disproof, Chowla~\cite{Chow} made the same conjecture which was again disproved by Heilbronn~\cite{Heilbr} a year later.

\subsection{Lower bounds on the number of real zeros of Fekete polynomials}

Let $N_D$ be the number of zeros of $F_D$ in $(0,1)$. Baker and Montgomery~\cite{BaMo} showed that, in fact, Fekete's conjecture fails in a strong form. More precisely, they proved that for every fixed integer $K\geq 1$, we have $N_D\geq K$ for almost all fundamental discriminants $D$.
Moreover they formulated the following guiding conjecture.
\begin{Con}\label{conjBM}\cite[Baker-Montgomery]{BaMo}
For almost all fundamental discriminants $D$, 
$$N_D \asymp \log \log |D|.$$
\end{Con}
 A similar conjecture for prime discriminants has also been mentioned by Conrey, Granville, Poonen and Soundararajan in~\cite{CGPS}. 
 Our first result comes close to establishing the implicit lower bound in Conjecture \ref{conjBM} in a stronger form, where we localize the number of zeros in a short
 interval close to $1$. Throughout, we let $\log_k(x)$ denote the $k$-th iterate of the natural logarithm.
 \begin{thm}
\label{MainFek}
Let $0<\alpha<1/20$ and $A\geq 1$ be  fixed constants. For all except for a set of size $O\big(x\exp(-(\log_3 x)^A)\big)$ fundamental discriminants $0<D\leq x$,  the corresponding Fekete polynomial $F_D$ has $$\gg_{\alpha,A} \frac{\log_2 D}{\log_4 D}$$ zeros in the interval  $\big(1-e^{-(\log D)^{\alpha/100}},1-e^{-(\log D)^{\alpha}}\big)$.
\end{thm}
\begin{rem} \
\begin{itemize}
\item Using the proof of the second part of Theorem \ref{Main} below, we could obtain a smaller exceptional set in Theorem \ref{MainFek}, at the cost of replacing the number of zeros by $(\log_2 D)/\log_3 D$ (see the second part of Theorem \ref{ThmPartialSumsPositive} below for an analogous statement).
\item It is natural to compare Conjecture \ref{conjBM} with the well-studied case of 
random polynomials $P_{n,\xi}(t):=\sum_{k=0}^{n} \xi_k t^k$ where $\xi_k$ are independent identically distributed copies of a real random variable $\xi$ with mean zero. 
Erd\H{o}s and Offord~\cite{EO} famously proved that in the Bernoulli case, i. e. with coefficients given by independent random variables taking the values $\pm 1,$ the expected number of real zeros is 
\begin{equation}\label{Kac}
 N(P)=\frac{2}{\pi}\log n + o (\log^{2/3} (n) \log \log n)
 \end{equation}
 with probability  $1+o_{n\to\infty}(1).$
 On the other hand, it is a rather challenging task to give upper and lower bounds for the number of real zeros for an individual deterministic Littlewood-type polynomial. We mention here a well-known result of Borwein, Erd\'elyi and K\'{o}s  \cite{ErdBor}, showing a sharp bound $O(\sqrt{n})$ for the number of zeros in $(-1,1)$ of 
any polynomial
$$p(z)=\sum_{j=0}^{n}a_j z^j, \hspace{2mm} \text{ such that }\vert a_j\vert \leq 1, \vert a_0\vert =  \vert a_n\vert =1, a_j \in \mathbb{C}.$$
\item  
It is widely believed 
that a 
Dirichlet character modulo $\vert D\vert$ is ``almost" determined by its values at the first $(\log |D|)^{1+o(1)}$ primes, 
and thus it seems conceivable, in line with Conjecture \ref{conjBM}, that the number of real zeros in this case could be substantially smaller than in the random case (compared with \eqref{Kac}).
\end{itemize}
\end{rem}

\subsection{Sign changes of character sums}

We now let $\F$ be the set of fundamental discriminants $|D|\leq x$. 
An important problem in analytic number theory is to gain an understanding of the statistical behaviour of the character sums
  $$ S_{\chi_D}(N)=  \sum_{1 \leq n \leq N} \chi_D(n).$$ For example, Granville and Soundararajan \cite{G-S-max-character}, Bober, Goldmakher, Granville and
Koukoulopoulos \cite{BGGK} and Lamzouri \cite{Lamzouri} investigated the behaviour of large character sums.
Harper \cite{Harper-short} recently showed that $S_{\chi}(N)$ typically exhibits below square root cancellations $o(\sqrt{N})$ in the family of complex characters $\chi$ modulo a fixed large prime.
An interesting recent paper by Hussain and Lamzouri \cite{AyeshaLamzouri} deals with the distribution of $S_{\chi_p}(tp)$ for a randomly selected $t\in [0,1]$, when $p$ varies over the primes in a dyadic interval $[x, 2x]$ and $x$ is large.\\

A well known conjecture of Vinogradov asserts that the first sign change in the sequence $\chi(1),\chi(2),\dots,\chi(|D|-1)$ should occur among the first $N\le |D|^{\varepsilon}$ terms. Our objective here is a more refined classical question which can be traced back to Chowla, Fekete and others.
\begin{center}
\textbf{\underline{Question}}: What can be said about sign changes in the sequence of partial sums $$\left\{S_{\chi_D}(1),S_{\chi_D}(2),\dots,S_{\chi_D}(|D|-1)\right\}?$$
\end{center}
Baker and Montgomery \cite{BaMo} proved that for $100\%$ of fundamental discriminants $|D|\le x,$ $S_{\chi_D}(N)< 0$ for at least one $0<N\le |D|-1.$
Exploring the connection to sign changes of partial sums $\sum_{n\leq x} f(n)$ of a Rademacher random multiplicative function, Kalmynin \cite{Kalmynin} showed that the relative proportion of prime discriminants $p \le x$ for which $S_{\chi_p}(N)\ge 0$ for {\it all } $0<N\le p-1$ is $\ll 1/(\log \log x)^{c}$ for $c\approx 0.03$ and conjectured that this proportion should be of size $\approx 1/(\log \log x)^{1-o(1)}$. See also \cite{Angelo-Xu}, \cite{Aymone} and \cite{Kerr-Klurman}
for recent related works on sign changes of partial sums of random multiplicative functions. Here we prove the following quantitative result. 
  \begin{thm}\label{ThmPartialSumsPositive}
  Let $0<\alpha<1/20$ and $A\geq 1$ be  fixed constants.
  \begin{itemize}
 \item[1.] For all except for a set of size $O\big(x\exp(-(\log_3 x)^A)\big)$ fundamental discriminants $|D|\leq x$,  $S_{\chi_D}(N)$ has at least $\gg_{\alpha, A} (\log_2
|D|)/\log_4|D|$ sign changes in the interval $e^{(\log |D|)^{\alpha/100}}\leq N\leq e^{(\log |D|)^{\alpha}}$. 

\item[2.] There exists a positive constant $c$ such that for all except for a set of size $O\big(x\exp(-c\log_2(x)/\log_3(x))\big)$ fundamental discriminants $|D|\leq x$,  $S_{\chi_D}(N)$ has at least $\gg_{\alpha, A} (\log_2
|D|)/\log_3|D|$ sign changes in the interval $e^{(\log |D|)^{\alpha/100}}\leq N\leq e^{(\log |D|)^{\alpha}}$.
\end{itemize}
\end{thm}
As an immediate consequence, we disprove the analogue of Kalmynin's conjecture \cite{Kalmynin} for fundamental discriminants.
\begin{cor}
There exists a positive constant $c$ such that the relative proportion of fundamental discriminants $|D| \le x$ for which $S_{\chi_D}(N)\ge 0$ for {\it all } $0<N\le |D|-1$ is 
$$ \ll \exp(-c\log_2(x)/\log_3(x)). $$
\end{cor}
An improved upper bound for this proportion, of the form $(\log x)^{-c_0}$ for some small positive constant $c_0$, has been announced to appear in a forthcoming work of Angelo, Soundararajan and Xu. We also note that by adapting our methods, one can obtain similar results in the case of prime discriminants (and hence disprove the original conjecture of Kalmynin in this case), assuming the non-existence of Landau–Siegel zeros.
\begin{rem}
\begin{itemize}
\item Existing results in the literature produce a bounded number of sign changes in the full range $1\le N\le |D|.$ The most important feature of our Theorem \ref{ThmPartialSumsPositive} compared to the previous works is the localization of  {\it many} sign changes in a {\it very short} interval which also goes beyond the range of Vinogradov's conjecture. The proof we give here is flexible and could produce sign changes on a shorter interval at the cost of reducing the number of such sign changes. Finally, we mention a result of Banks, Garaev, Heath-Brown and Shparlinski \cite{BGHS} who showed that there is a positive proportion of quadratic non-residues in the Burgess' range $(N\geq D^{\frac{1}{4\sqrt{e}}+\varepsilon})$.
\item In a closely related random multiplicative case (the Rademacher model), Aymone, Heap and Zhao \cite{Aymone} exhibit almost surely an infinite number of sign changes. Building on \cite{Aymone}, very recently Geis and Hiary \cite{G-H} showed that almost surely the number of sign changes of $\sum_{n\leq t}f(n)$ for $t\in[0,x]$ is at least $(\log\log\log x)^{1/c}$ for some constant $c>2,$ where $f(n)$ is a Rademacher random multiplicative function. Following our proof of Theorem \ref{ThmPartialSumsPositive} and replacing averaging over the characters by the expectations one can improve these bounds to get the following corollary.
\begin{cor}
Let $f:\mathbb{N}\to \{-1,0,1\}$ be a Rademacher random multiplicative function. Then almost surely for large $x\ge 1$, the partial sums 
$\sum_{n\le y}f(n)$ exhibit $\gg\log_2 x/\log_4 x$ sign changes on the interval $[1,x].$
\end{cor}
We remark that these sign changes might not be detected as locally as in the work of Geis and Hiary.
\end{itemize}
 \end{rem}

The proofs of Theorems \ref{MainFek} and \ref{ThmPartialSumsPositive}, in turn, rely on appropriate variants of the following quantitative estimate concerning oscillations of $L'(s,\chi_D)$, coupled with a concentration result for the distribution of $L(s,\chi_D)$ in the vicinity of $1/2$ (see Proposition \ref{CLTLogL} below).
\begin{thm}\label{Main} 
Let $A\geq 1$ be a fixed constant.
\begin{itemize}
\item[1.]  For all except for a set of size $O\big(x\exp(-(\log_3 x)^A)\big)$ fundamental discriminants $|D|\leq x$, $L'(s, \chi_D)$ has $\gg_A (\log_2|D|)/\log_4|D|$ real zeros in the interval $(0, 1)$.

\item[2.] There exists a positive constant $c$ such that for all except for a set of size $O\big(x\exp(-c\log_2(x)/\log_3(x))\big)$ fundamental discriminants $|D|\leq x$,  $L'(s, \chi_D)$ has $\gg (\log_2|D|)/\log_3|D|$ real zeros in the interval $(0, 1)$. \end{itemize}
\end{thm}
We mention yet another application of Theorem \ref{Main}.
Let $D$ be a positive fundamental discriminant and let
\begin{equation}\label{theta} \theta(t,\chi_D):=\sum_{n=1}^{\infty} \chi_D(n)e^{-\frac{\pi n^2t}{D}} \,\,\,(t>0) \end{equation} be the theta function associated to $L(s,\chi_D),$ (see \cite[Chapter 9]{Dav}). The study of real zeros of $\theta(\cdot,\chi_D)$ was initiated in \cite{Bengo}, \cite{BMT}, \cite{Debrecen}, \cite{jnttheta} and is also mentioned in \cite{Harper-short}, \cite{HarperMaks}, \cite{Harperhigh} and \cite{SarBach}. Theorem \ref{Main} immediately implies the following.\footnote{We did not aim to localize the zeros of $\theta(\cdot,\chi_D)$ as in Theorem \ref{MainFek} or to quantify the size of the exceptional set, but this should follow using the same methods.}
  \begin{cor}\label{zerostheta} 
  For almost all 
  fundamental discriminants $0< D\leq x$, $\theta(\cdot,\chi_D)$  has $\gg (\log_2D)/\log_4D$ real zeros in the interval $(0, \infty)$.
  \end{cor}

\subsection{Upper bounds for the number of real zeros of Fekete polynomials.}
As was already remarked, it is a rather challenging task to give upper bounds for the number of real zeros of an individual Littlewood-type polynomial. 
In a recent survey, Erd\'elyi~\cite{ErdSurv} suggests that the only known upper bound for the number of zeros of $F_D$ is $O(\sqrt{|D|}),$ which follows from a general upper bound valid for any Littlewood polynomial proved by Borwein, Erd\'elyi and K\'{o}s  \cite{ErdBor}.
 In Section \ref{upperbounds}, we provide a ``non-trivial" upper bound for the number of real zeros of Fekete polynomials on average over discriminants. Our result in this direction is the following.
\begin{thm}\label{Thmupperbnd}
 For at least $\gg x^{1-o(1)}$ fundamental discriminants $0<D\leq x$, the associated Fekete polynomial $F_D$ has at most $O(x^{1/4+o(1)})$ real zeros. 
\end{thm}
\subsection{Atypical behavior: Fekete's hypothesis and a localized version}
After the disproof (twice!) of Fekete's conjecture, a more plausible statement has been put forward:\\

\underline{\textbf{Fekete's hypothesis}}: There exists infinitely many fundamental discriminants $D$ such that $F_D$ has no zeros in $(0,1)$. \\

This condition of ``positive definiteness'' is a $GL_1$ manifestation of a more general phenomenon in the context of $GL_n$ automorphic representations (see  \cite{Jung} and \cite{SarBach}). In \cite{SarBach}, the heuristic in support of this hypothesis is based on an application of the Kac-Rice formula and uses the result of Dembo, Poonen and Zeitouni \cite{Poonen} (showing that 
the probability that a random polynomial of large even degree $n$ has no real zeros is $n^{-b+o(1)}$ for some constant $0.4 \leq b \le 2 $). However such heuristic works for general random coefficients and does not take into account the multiplicativity of $\chi_D.$ \\

An alternative way is to note that Descartes' rule of signs implies that if the partial character sums $S_{\chi_D}(N)$ remain positive for all $0<N \leq \vert D\vert $ then $F_D$ has no root in $(0,1)$ (see for instance the figure above). Of course, it follows from the result of Kalmynin \cite{Kalmynin} mentioned above and our Theorem \ref{ThmPartialSumsPositive} that this is a rare event, though numerical computations furnish evidence that it occurs with reasonably high probability. Roughly speaking, if one arranges $\chi_D(p)=1$ for many small primes $p\le y=\exp(\log^{1/100} D),$ say (in which case $\chi$ becomes ``$1$-pretentious" in the sense of Granville and Soundararajan \cite{G-S-max-character}), then one can expect the character sums $S_{\chi_D}(x)$
to be dominated by the contribution of the $y-$smooth part $\sum_{n\le x,   n \text{ is 
$y$-smooth}}1,$ for all $y<x<D,$ provided there is some ``randomness" for the values of $\chi_D(p)$ for $p>y.$ Establishing such a randomness rigorously seems hard.
\begin{figure}
\centering
\includegraphics[scale=0.5]{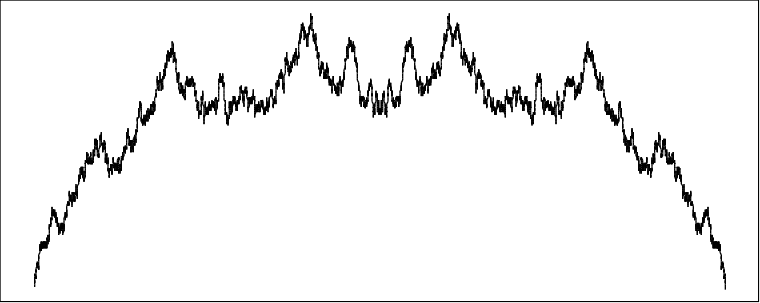}
\caption{Partial sums $S_{\chi_{D}}(N), N=1\dots D-1$ for $D=7727$.} 
\label{figpos}
\end{figure}
Following the latter heuristic and combining with Vinogradov's trick, we exhibit many fundamental discriminants $D$ such that the associated Fekete polynomial does not have real zeros away from $z=1$.
\begin{thm}\label{NoZerosThm}
Let $\varepsilon>0$ be a fixed small number, and $x$ be large. There exists at least $x^{1-1/\log\log x}$ fundamental discriminants $0<D\leq x,$ such that $F_D(z)$ has no zeros in the interval $\left(0, 1-(\log x)^{-\sqrt{e}+\varepsilon}\right)$.
\end{thm}


\section{Outline of the proofs}
We shall briefly outline the key flexible idea to produce sign changes of the partial sums $S_f(x)=\sum_{n\le x}f(n)$ of multiplicative functions. This ``soft" approach would work equally well if $f=\chi_D$ is replaced by other families of multiplicative functions, 
 for example by a Rademacher random multiplicative function, or by the normalized Fourier coefficients of automorphic forms, producing similar quantitative results.\\
By partial summation we can write 
$$ F(x,s)=\sum_{n\le x}\frac{f(n)}{n^s}= s\int_1^{x} \frac{S_{f}(u)}{u^{1+s}}du.$$
Now dividing by $s$ and taking the derivatives of both sides with respect to $s$ we obtain
$$ \frac{sF'(x,s)-F(x,s)}{s^2}= -\int_1^{x} \frac{S_{f}(u)\log u}{u^{1+s}}du.$$
Making the change of variables $t=\log u$ in the above integral gives the identity
\begin{equation}\label{Eq.LaplaceCharactersintro} \frac{F(x,s)}{s} \left(-\frac{F'}{F}(x,s)+\frac{1}{s}\right)
= \int_0^{\log x} t S_{f}(e^t) e^{-st}dt.
\end{equation}
A version of a general result of Karlin (see Lemma \ref{lem:Karlin} below) forces the quantity $tS_{f}(e^t)$ to have at least as many sign changes on the interval $[0,\log x)$ as the LHS of \eqref{Eq.LaplaceCharactersintro} (its Laplace transform) for $1/2<s<1.$ We are thus naturally lead to understand the sign changes of $-\frac{F'}{F}(x,s).$ The key observation now is that this logarithmic derivative is typically well approximated by a sum over primes, which allows us to use probabilistic methods and compare it to a suitable random model, constructed as a weighted sum of independent random variables.
\subsection{Proof ideas of Theorems \ref{MainFek}, \ref{ThmPartialSumsPositive}  and  \ref{Main}.} 
First, for the proof of Theorem \ref{ThmPartialSumsPositive}, the preceding strategy takes a particularly elegant form.  We recall that for $\Re(s)>0$ we have by partial summation
$$ L(s, \chi_D)= s\int_1^{\infty} \frac{S_{\chi_D}(u)}{u^{1+s}}du.$$
Following the discussion above we arrive at
\begin{equation}\label{Eq.LaplaceCharactersintro2} \frac{L(s, \chi_D)}{s} \left(-\frac{L'}{L}(s, \chi_D)+\frac{1}{s}\right)
= \int_0^{\infty} t S_{\chi_D}(e^t) e^{-st}dt.
\end{equation}
We are thus naturally lead to understand the distribution of $-\frac{L'}{L}(s, \chi_D).$\\
Our proof of Theorem \ref{Main} is inspired by a beautiful work of Baker and Montgomery \cite{BaMo}.
In order to establish sign changes of $-\frac{L'}{L}(s, \chi_D)$ 
we first show that for a fixed point $s=s(x),$
 $s=1/2+ 1/g(x)$ where $g(x)\to \infty$ slower than a small power of $\log x$, for almost all $D\in \F$, we can approximate $-\frac{L'}{L}(s, \chi_D)$ by short Dirichlet polynomials $\mathcal{P}_D(s)$ over primes that are around $e^{g(x)}$. For a fixed $s,$ with appropriately chosen parameters, such short Dirichlet polynomials are well approximated by Gaussian random variables $X_s$ with mean zero and variance of size $\frac{1}{2s-1}.$ We shall choose many points $s_1,s_2\dots, s_R,$ where $R=R(x)$ in such a way that the supports of the  polynomials $\mathcal{P}_D(s_i)$ for $1\le i\le R$ are disjoint, which forces the corresponding Gaussian random variables to be independent.
It is crucial for our applications to have strong uniformity in all the parameters and we use the methods developed by Lamzouri, Lester and Radziwill \cite{LLR2}, \cite{LLR1} to bound the ``discrepancy" between the distribution of the vector of Dirichlet polynomials  $\left(\mathcal{P}_D(s_i)\right)_{i=1, \dots, R}$ and the associated random vector $(X_{s_i})_{i=1, \dots, R}$ which, with high probability, has a positive proportion of sign changes.

The starting point of the proof of Theorem \ref{MainFek} is the identity (derived from \eqref{Dirichlet-identity}),
\begin{equation}\label{Eq.IdentityLaplaceFekete}
L(s,\chi_D)\Gamma(s)\left(\frac{L'(s,\chi_D)}{L(s,\chi_D)}+\frac{\Gamma'(s)}{\Gamma(s)}\right)=
\int_{0}^{\infty} F_D(e^{-t})(1-e^{-|D|t})^{-1}t^{s-1}(\log t)dt.
\end{equation}
We remark that since we are aiming for quantitatively strong results on short intervals, this causes considerable difficulties. In particular, this prevents us from applying Descartes rule of signs to immediately deduce the bounds on the number of sign changes of partial sums $\sum_{n\le x}\chi_D(n)$ from the corresponding bound on the zeros of Fekete polynomials (and so Theorem \ref{ThmPartialSumsPositive} is not a consequence of Theorem \ref{MainFek}).\\
In order to localise the sign changes in \eqref{Eq.LaplaceCharactersintro2} and \eqref{Eq.IdentityLaplaceFekete}
we shall select points $s_1,s_2\dots, s_R$ to guarantee that 
\[ \pm\left(-\frac{L'(s_r,\chi_D)}{L(s_r,\chi_D)}+\frac{1}{s_r}\right)\gg  \frac{1}{2s_r-1}\] and similarly (with slightly different choice of sampled points),
\[ \pm\left(\frac{L'(s_r,\chi_D)}{L(s_r,\chi_D)}+\frac{\Gamma'(s_r)}{\Gamma(s_r)}\right)\gg  \frac{1}{2s_r-1}\]
for every $1\le r\le R$ and additionally require $L(s_r,\chi_D)\gg 1/(2s_r-1)^{1/4}$ for $100\%$ of $D\in\mathcal{F}(x).$ This is accomplished by proving a large deviation result for the distribution of $\log L(s_r,\chi_D),$ $D\in\mathcal{F}(x)$ (see Proposition \ref{CLTLogL} below). We thus guarantee that for ``good'' discriminants the RHS of \eqref{Eq.LaplaceCharactersintro2} and \eqref{Eq.IdentityLaplaceFekete} have size $\gg (2s_r-1)^{-1}$. Our goal then is to show that for almost all $D\in\mathcal{F}(x)$, we can uniformly truncate the integrals on the RHS by showing that the main contribution comes from the localised quantities
\[\int_{Y}^{Z} t S_{\chi_D}(e^t) e^{-st}dt\]
and
$$\int_{e^{-Z}}^{e^{-Y}} F_D(e^{-t})(1-e^{-\vert D\vert t})^{-1}t^{s_r-1}(\log t) dt$$
for appropriately chosen parameters $Y, Z.$ We shall achieve this by a careful analysis of the underlying weighted character sums, using a variety of analytic tools and large sieve inequalities and taking advantage of the averaging over $D\in\mathcal{F}(x).$ As an outcome of this procedure, we produce sign changes 
for many selected points $s_1,s_2\dots s_R,$ which forces the integrands to change sign in this range.\\
\subsection{ Proof ideas of Theorem \ref{Thmupperbnd}.} 
Our approach proceeds by covering the segment $[0,1)$ by several small circles $\mathcal{C}(z_{\alpha},r_{\alpha})$ centered at the points 
$z_{\alpha}:=\exp(-1/x^{\alpha})$ (for different values of $\alpha$) and  bound  the number of (complex) zeros inside each $\mathcal{C}(z_{\alpha},r_{\alpha})$  using Jensen's formula. This simple idea goes back to Littlewood and Offord \cite{LO1,LO2} in their fundamental work on the number of real roots of random polynomials. In doing so, we need to ensure that for many discriminants  $D\in\mathcal{F}(x),$  the associated polynomial $F_D$ is simultaneously not too small at several points $z_{\alpha_i}$, $1\leq i \leq I$, say
$$|F_D(z_{\alpha_i})|\gg \frac{1}{x^{100}}, \qquad \qquad 1\leq i \leq I$$ for some $ I \geq 1.$ To prove such a result, it is sufficient to  give lower and upper bounds on the mixed moments 
$$S_k(\alpha_1,\cdots,\alpha_I)=\sum_{D\in \F} \prod_{i=1}^I F_D(z_{\alpha_i})^k \hspace{2mm} \textup{for  } k\in \mathbb{N}.$$ Working with real characters instead of random coefficients $\pm 1$ is a more subtle task which forces the number of sampled points to be small. 
We pursue this strategy bounding the first two moments for three well-chosen points ($I=3$ and $k=1,2$) by observing that $F_D(z_{\alpha})$ resembles a character sum of length $\approx D^{\alpha}$. Therefore, we can use standard techniques, including the P\'olya-Vinogradov inequality, large sieve inequalities and the Poisson summation formula, in order to bound averaged character sums in various ranges. 

\subsection{Notation and conventions}


For any integer $k\geq 2$, we use $\tau_k(n)$ and $\varphi(n)$ to denote the $k$-th generalized divisor function and 
the Euler function, of an integer $n \ge 1$, respectively. We also use the standard notation $\tau(n)=\tau_2(n).$  It is useful to recall the well-known estimates
\begin{equation*}
\tau_k(n) = n^{o(1)} \ \text{ for fixed } k,\qquad\mbox{and}\qquad  \varphi(n) \gg \frac{n}{\log \log n}
\end{equation*}
as $n\to \infty$, which we use throughout the paper.


\section{Mean values of real character sums }

In this section we record several mean value estimates and large sieve inequalities for real character sums which will be repeatedly used throughout.  
The first is an ``orthogonality relation'' for quadratic characters, which is a simple application of the P\'olya-Vinogradov inequality.
\begin{lem}[Lemma of 4.1 of \cite{GrSo}]\label{Orthogonality}
For all positive integers $n$ we have 
$$ \sum_{D\in \F} \chi_D(n)\ll x^{1/2}n^{1/4}\log n,$$
if $n$ is not a perfect square. On the other hand if $n=m^2,$ then 
$$ \sum_{D\in \F} \chi_D(n)=\frac{6}{\pi^2}x\prod_{p|m}\left(\frac{p}{p+1}\right)+ O\big(x^{1/2} \tau(m)\big).$$
\end{lem}
Note that by taking $n=1$ in Lemma \ref{Orthogonality} we get 
\begin{equation}\label{AsympF}
|\F|=\frac{6}{\pi^2}x + O\big(\sqrt{x}\big).
\end{equation}

With Lemma \ref{Orthogonality} at our disposal, we prove the following large sieve type result.

 \begin{lem}\label{LargeSieve}
Let $\{a(p)\}_{p}$ be a sequence of real numbers indexed by the primes. Let $x$ be large and $2\leq y\leq z$ be real numbers. Then for all positive integers $k$ such that $1\leq k\leq \log x/(5\log z)$ we have
\begin{equation}\label{LargeSieve1}
\sum_{D\in \F}\Big|\sum_{y\leq p\leq z}a(p)\chi_D(p)\Big|^{2k}\ll x\left(k\sum_{y\leq p\leq z}a(p)^2\right)^k+x^{5/8}\left(\sum_{y\leq p\leq z}|a(p)|\right)^{2k}.
\end{equation}
In particular, uniformly for $\sigma\geq 1/2$ we have 
\begin{equation}\label{LargeSieve2}
\sum_{D\in \F}\left|\sum_{y\leq p\leq z}\frac{\chi_D(p)\log p}{p^{\sigma}}\right|^{2k}\ll x\left(C_1 k (\log z)^2\right)^k,
\end{equation} 
for some absolute constant $C_1>0$. 

\end{lem}

\begin{proof}
We start by establishing \eqref{LargeSieve1}. We have 
$$ \sum_{D\in \F}\Big|\sum_{y\leq p\leq z}a(p)\chi_D(p) \Big|^{2k}= \sum_{D\in \F}\sum_{y\leq p_1,...,p_{2k}\leq z}a(p_1) \cdots a(p_{2k})\chi_d(p_1...p_{2k}).$$
The diagonal terms $p_1...p_{2k}=\square$ contribute
$$ \ll x\frac{(2k)!}{2^k k!}\left(\sum_{y\leq p\leq z}a(p)^2\right)^k\leq x\left(k\sum_{y\leq p\leq z}a(p)^2\right)^k.$$
To handle the off-diagonal terms we use Lemma \ref{Orthogonality}. Thus if $p_1p_2...p_{2k}\neq \square$ and $p_i\leq z$ then Lemma \ref{Orthogonality} gives
$$
\sum_{D\in \F}\chi_d(p_1p_2...p_{2k})\ll x^{1/2}z^{k/2}\log x\ll x^{5/8},
$$
which implies that the contribution of these terms is
$$ \ll x^{5/8}\left(\sum_{y\leq p\leq z}|a(p)|\right)^{2k}.$$
This completes the proof of \eqref{LargeSieve1}. Now, to prove \eqref{LargeSieve2}, we choose $a_p= (\log p)/p^{\sigma}$ in \eqref{LargeSieve1}. This gives 
\begin{align*}
\sum_{D\in \F}\left|\sum_{y\leq p\leq z}\frac{\chi_D(p)\log p}{p^{\sigma}}\right|^{2k} 
& \ll x\left(k\sum_{y\leq p\leq z}\frac{(\log p)^2}{p}\right)^k+x^{5/8}\left(\sum_{y\leq p\leq z}\frac{\log p}{\sqrt{p}}\right)^{2k}\\
& \ll x\left(C_1 k (\log z)^2\right)^k,
\end{align*}
for some positive constant $C_1$, since
$\sum_{y\leq p\leq z} (\log p)/\sqrt{p} \ll \sqrt{z}\ll  x^{1/(10k)}.$
\end{proof}

We will need the following results on the mean-square value of real character sums, which are due to Jutila \cite{Jut1}.

\begin{lem}[Theorem 1 and the Corollary of \cite{Jut1}]\label{lemmaJutila}
Let $x\geq 3$ and $N\geq 1$. Then 
\begin{equation}\label{lem.Jutila}
  \sum_{D\in \F} \Big|\sum_{n\leq N}  \chi_D(n)\Big|^2 \ll xN(\log x)^8. 
\end{equation}
Moreover, we have 
\begin{equation}\label{Jutila}
    \sum_{\substack{n\leq N \\ n\neq \square} } \Big|\sum_{D\in \F}\chi_D(n)\Big|^2 \ll xN (\log N)^{10},
\end{equation} where the outer summation runs over integers which are not perfect squares.
\end{lem}


We end this section by recording several large sieve inequalities for real characters. We start with the two standard estimates which are straightforward applications of the P\'olya-Vinogradov inequality.

\begin{lem}\label{BabyLargeSieve}
Let $x, N\geq 2$. Then for arbitrary complex numbers $a_n$ we have 
\begin{equation}\label{Eq.BabyLargeSieve}
\sum_{D\in \F} \Big|\sum_{n\leq N} a_n \chi_D(n)\Big|^2 \ll (x+ N^2 \log N) \sum_{\substack{m, n \leq N\\ mn= \square}} |a_ma_n|,
\end{equation}
and
\begin{equation}\label{Eq.BabyLargeSieve2}
\sum_{D\in \F} \Big|\sum_{n\leq N} a_n \chi_D(n)\Big|^2 \ll x \sum_{\substack{m, n \leq N\\ mn= \square}} |a_ma_n| + \log N \left(\sum_{n\leq N} |a_n|n^{1/2}\right)^2.
\end{equation}
\end{lem}
\begin{proof} The first inequality corresponds to Lemma 1 of \cite{BaMo}. The second is a variant which can be found in Lemma 4 of \cite{Armon}.
\end{proof}

Next we state Heath-Brown's famous large sieve inequality for real characters. 
\begin{lem}\cite[Corollary $2$]{HBsieve}\label{HB-largesieve}
Let $x,N \geq 2$. Then for arbitrary complex numbers $a_n$ and any $\varepsilon>0$ we have 
$$  \sum_{D \in \F} \Big\vert \sum_{n \leq  N} a_n \chi_D(n)\Big\vert^2 \ll_{\varepsilon} (xN)^{\varepsilon}(x+N)\sum_{\substack{m,n \leq N \\ mn=\square}} \vert a_m a_n\vert. 
$$

\end{lem}


\section{Approximating $\log L(s, \chi_D)$ and $L'/L(s, \chi_D)$ in the vicinity of the central point}
\subsection{Short Dirichlet approximations of $-L'/L(s, \chi_D)$ close to the central point}
In this section we prove the following proposition, which shows that if $s=1/2+ 1/g(x)$ where $g(x)\to \infty$ in a certain range in terms of $x$, 
then for almost all $D\in \F$, we can approximate $-L'/L(s, \chi_D)$ by short Dirichlet polynomials over primes that are around $e^{g(x)}$. This will be one of the key ingredients in the proof of Theorem \ref{Main}. 
\begin{pro} \label{TruncationPrimes}
Let $x$ be large, and $s=1/2+ 1/g(x)$ where $(\log\log x)^2\leq g(x) \leq \sqrt{\log x}/(\log\log x)^2$. Let $2\leq M\leq \log\log x$ be a parameter and put 
$$u(s)= \exp\left(\frac{g(x)}{M}\right), \text{ and } v(s)= \exp\big( g(x) M\big).$$ 
The number of fundamental discriminants $|D|\leq x$ such that 
$$ 
\left|\frac{L'}{L}(s,\chi_D)+ \sum_{u(s) < p < v(s)} \frac{\chi_D(p)\log p}{p^s} \right| >  g(x), 
$$
is 
$$
\ll  x \exp\left(-\frac{M^2}{60}\right).
$$
\end{pro}

To prove this result, we need the following standard lemma which is established using zero density estimates for the family $\{L(s, \chi_D)\}_{D\in \F}$, and follows for example from the proof of Lemma 2 of \cite{BaMo}.  
 
\begin{lem}\label{ZeroDensity}\label{ApproximationLarge} Let $x$ be large and $1/2+ (\log\log x)^2/\log x \leq s \leq 1$. Put $A=12/(s-1/2)$ and let $(\log |D|)^A\leq y\leq |D|$ be a real number.  Then for all fundamental discriminants $|D|\leq x$ except for a set $\mathcal{E}(x)$ with cardinality 
$$ |\mathcal{E}(x)| \ll x^{1-(s-1/2)/5},$$ we have 
\begin{equation}\label{ApproxLprimeL} -\frac{L'}{L}(s,\chi_D)=\sum_{n=1}^{\infty} \frac{\Lambda(n)}{n^s}\chi_D(n)e^{-n/y}+O\left(\frac{1}{\log |D|}\right),
\end{equation}
and 
\begin{equation}\label{ApproxLogL}\log L(s,\chi_D)=\sum_{n=1}^{\infty} \frac{\Lambda(n)}{(\log n)n^s}\chi_D(n)e^{-n/y}+O\left(\frac{1}{\log |D|}\right).
\end{equation}
\end{lem}
\begin{proof}
The asymptotic formula \eqref{ApproxLprimeL} corresponds to Lemma 2 of \cite{BaMo}. The formula \eqref{ApproxLogL} is obtained by a slight modification of the same proof. 
\end{proof}

\begin{proof} [Proof of Proposition \ref{TruncationPrimes}]
Let $y= (\log x)^{12 g(x)}$. First,  Lemma \ref{ZeroDensity} implies that for all but at most $O(x^{1-1/(5g(x))})$ fundamental discriminants $|D|\leq x$ we have 
\begin{equation}\label{FirstTrunc}
\left|\frac{L'}{L}(s,\chi_D)+ \sum_{n=1}^{\infty} \frac{\chi_D(n)\Lambda(n)}{n^s}e^{-n/y}\right| \ll \frac{1}{\log |D|}.
\end{equation}
The contribution of the terms $n>y^2$ to the sum above is 
\begin{equation}\label{TailBoundLPrimeL}
\ll \sum_{n>y^2} \frac{\log n}{\sqrt{n}} e^{-n/y}\ll e^{-y/2}\sum_{n>y^2} \frac{\log n}{\sqrt{n}} e^{-n/(2y)} \ll e^{-y/2} \sum_{n=1}^{\infty} \frac{1}{n^2} \ll \frac{1}{\log x},
\end{equation}
since $e^{n/(2y)} \gg n^3$ if $n>y^2$, and $y$ is sufficiently large. In particular, this implies that the contribution of prime powers $p^k$ with $k\geq 3$ to the sum in \eqref{FirstTrunc} is 
\begin{equation}\label{LargePrimePowers}
 \ll  \sum_{k\geq 3} \sum_{p \leq y} \frac{\log p}{p^{k/2}}+ \frac{1}{\log x} \ll 1.
 \end{equation}
Moreover, it follows from the prime number theorem and partial summation that the contribution of the squares of primes to the sum in \eqref{FirstTrunc} is 
\begin{equation}\label{SquaresPrimes}
\sum_{p\nmid D} \frac{(\log p) e^{-p^2/y}}{p^{2s}}\leq \sum_{p} \frac{\log p}{p^{2s}}=(1/2+o(1))g(x).
\end{equation}
Now let $\mathcal{E}_1(x)$ be the set of fundamental discriminants $|D|\leq x$ such that 
$$
\left|\frac{L'}{L}(s,\chi_D)+ \sum_{p\leq y^2} \frac{\chi_D(p)\log p}{p^s}e^{-p/y}\right| > \frac{5}{8} g(x).
$$ 
Then combining the estimates \eqref{LargePrimePowers} and \eqref{SquaresPrimes} with \eqref{FirstTrunc} shows that 
\begin{equation}\label{Noise}
|\mathcal{E}_1(x)| \ll x \exp\left(-\frac{\log x}{5 g(x)}\right).
\end{equation}
Let $\mathcal{E}_2(x)$ be the set of fundamental discriminants $|D|\leq x$ such that 
$$
\left| \sum_{p\leq u(s)} \frac{\chi_D(p)\log p}{p^s}e^{-p/y}\right| > \frac{g(x)}{6}.
$$ 
Then, it follows from Lemma \ref{LargeSieve} that for all positive integers $k$ such that $k\ll  \sqrt{\log x}$ we have 
\begin{equation}\label{SmallPrimes0}
\begin{aligned}
|\mathcal{E}_2(x)| 
& \leq \left(\frac{6}{g(x)}\right)^{2k} \sum_{D\in \F}\left| \sum_{p\leq u(s)} \frac{\chi_D(p)\log p}{p^s}e^{-p/y}\right|^{2k}\\
& \ll x\left( \frac{36k}{g(x)^2} \sum_{p\leq u(s)} \frac{(\log p)^2}{p^{2s}} e^{-2p/y}\right)^k + x^{5/8}\left(\sum_{p\leq u(s)}\frac{\log p}{\sqrt{p}}\right)^{2k}\\
& \ll x \left(\frac{20k(\log u(s))^2}{g(x)^2}\right)^k+ x^{3/4} \ll x \left(\frac{20 k}{M^2}\right)^k\end{aligned}
\end{equation}
since
$$ \sum_{p\leq u(s)} \frac{(\log p)^2}{p^{2s}} e^{-2p/y}\leq \sum_{p\leq u(s)} \frac{(\log p)^2}{p} =(1/2+o(1)) (\log u(s))^2, 
$$ 
by the prime number theorem. Thus, choosing  $k= \lfloor M^2/60\rfloor$, we deduce that 
\begin{equation}\label{SmallPrimes}
|\mathcal{E}_2(x)| \ll x \exp\left(-\frac{M^2}{60}\right).
\end{equation}
Furthermore, let $\mathcal{E}_3(x)$ be the set of fundamental discriminants $|D|\leq x$ such that 
$$
\left| \sum_{ v(s)\leq p\leq y^2} \frac{\chi_D(p)\log p}{p^s}e^{-p/y}\right| > \frac{g(x)}{6}.
$$ 
Similarly to \eqref{SmallPrimes0} we deduce from Lemma \ref{LargeSieve} that for all positive integers $k$ such that $k\ll  \sqrt{\log x}$ we have 
\begin{align*}
|\mathcal{E}_3(x)| 
& \leq \left(\frac{6}{g(x)}\right)^{2k} \sum_{D\in \F}\left| \sum_{v(s)\leq p\leq y^2} \frac{\chi_D(p)\log p}{p^s}e^{-p/y}\right|^{2k}\\
& \ll x\left( \frac{36k}{g(x)^2} \sum_{v(s)\leq p\leq y^2} \frac{(\log p)^2}{p^{2s}} e^{-2p/y}\right)^k + x^{5/8}\left(\sum_{p\leq y^2}\frac{\log p}{\sqrt{p}}\right)^{2k} \\
&
\ll x\left(\frac{36k}{g(x)^2} \sum_{v(s)\leq p\leq y^2} \frac{(\log p)^2}{p^{2s}} \right)^k +x^{3/4}.
\end{align*}
Now by the prime number theorem and partial summation, we have 
$$  \sum_{v(s)\leq p\leq y^2} \frac{(\log p)^2}{p^{2s}} \ll \frac{\log v(s)}{v(s)^{2s-1}(s-1/2)} + \frac{1}{(s-1/2)^2 v(s)^{2s-1}}
\ll  g(x)^2Me^{-2M}.
$$
Thus, making the same choice $k= \lfloor M^2/60\rfloor$ gives 
\begin{equation}\label{LargePrimes}
|\mathcal{E}_3(x)| \ll x \exp\left(-\frac{M^2}{60}\right).
\end{equation}
Finally using that $e^{-p/y}= 1+O(p/y)$ we get 
$$ \sum_{u(s)\leq  p\leq v(s)} \frac{\chi_D(p)\log p}{p^s}e^{-p/y}- \sum_{u(s)\leq p\leq  v(s)} \frac{\chi_D(p)\log p}{p^s} \ll \sum_{p\leq v(s)} \frac{\sqrt{p}\log p}{y} \ll \frac{v(s)^{3/2}}{y} \ll 1.$$
Combining this estimate with \eqref{Noise}, \eqref{SmallPrimes}, and \eqref{LargePrimes} completes the proof.
\end{proof}
\subsection{The typical size of $\log L(s,\chi_D)$ near the central point}
The following important concentration result for $L(s,\chi_D)$ implies in particular that $L(s,\chi_D)$ is typically  not too small in the vicinity of $s=1/2$. This is used in the proofs of Theorems \ref{MainFek} and \ref{ThmPartialSumsPositive} to show that the factor $L(s,\chi_D)$ does not affect the size of $-\frac{L'}{L}(s, \chi_D)$ for almost all discriminants $D\in \F.$
\begin{pro}\label{CLTLogL}
Let $0<\alpha<1$ be fixed. Let $x$ be large and $1/2+ (\log\log x)^2/\log x\le s\leq 1/2+1/(\log x)^{\alpha}$ be a real number. There exists positive constants $C_1$ and $C_2$ such that uniformly for $V$ in the range $1\leq V \leq C_1 \sqrt{ (s-1/2)\log x/\log\log x}$, the number of fundamental discriminants $D\in \mathcal{F}(x)$ such that 
$$\left|\log L(s,\chi_D)- \frac{1}{2}\log\left(\frac{1}{s-1/2}\right)\right|> V \sqrt{\log \left(\frac{1}{s-1/2}\right)}
$$ 
is 
$$ 
\ll x \exp\left(-C_2V^2\right).
$$
\end{pro}
\begin{proof}
Let $A= 12/(s-1/2)$ and $y= (\log x)^A$. Let $\mathcal{E}(x)$ be the exceptional set in Lemma \ref{ZeroDensity}. Then, for all fundamental discriminants $|D|\leq x$ with $D\notin \mathcal{E}(x)$ we have 
\begin{equation}\label{FirstApproxLogL}
\log L(s,\chi_D)= \sum_{p} \frac{\chi_D(p)}{p^s}e^{-p/y} + \frac{1}{2}\sum_{p\nmid D} \frac{e^{-p^2/y}}{p^{2s}} +O(1).
\end{equation}
We first estimate the contribution of the terms $p^2$ to $\log L(s, \chi_D)$. Let $p_j$ denote the $j$-th prime, and $\omega(D)\leq \log |D|$ denote the number of prime factors of $D$. Note that 
$$  \sum_{p\mid D} \frac{e^{-p^2/y}}{p^{2s}} \leq \sum_{p\mid D} \frac{1}{p}\leq \sum_{p\leq p_{\omega(D)}}\frac{1}{p}\ll \log\log p_\omega(D)\ll \log\log\left(\frac{1}{s-\frac12}\right), $$
by our assumption on $s$ and since $p_j\ll j\log j$. Furthermore, we have 
$$ 
\sum_{p} \frac{1-e^{-p^2/y}}{p^{2s}}\ll \frac{1}{y}\sum_{p\leq \sqrt{y}} \frac{p^2}{p^{2s}}+ \sum_{p>\sqrt{y}} \frac{1}{p^{2s}} \ll \frac{1}{y}\sum_{p\leq \sqrt{y}} p+ \frac{1}{s-1/2} y^{1/2-s}\ll 1,
$$
by the prime number theorem and our assumption on $y$ and $s$. Combining these estimates we deduce that the contribution of the squares of primes to $\log L(s, \chi_D)$ is 
\begin{equation}\label{ContributionSquaresLogL}
\begin{aligned}
\frac{1}{2}\sum_{p\nmid D} \frac{e^{-p^2/y}}{p^{2s}}
&= \frac{1}{2}\sum_{p} \frac{1}{p^{2s}}+ O\left(\log\log\left(\frac{1}{s-\frac12}\right)\right)\\
&= \frac{1}{2}\log\left(\frac{1}{s-1/2}\right)+O\left(\log\log\left(\frac{1}{s-\frac12}\right)\right),
\end{aligned}
\end{equation}
by the prime number theorem. On the other hand by \eqref{TailBoundLPrimeL} the  contribution of the primes $p>y^2$ to the first sum on the right hand side of \eqref{FirstApproxLogL} is
$$ \ll \sum_{p>y^2} \frac{1}{p^s}e^{-p/y}\ll  \frac{1}{\log x}.$$
Let $k\leq \log x/(10 \log y)$ be a positive integer to be chosen. Combining this last estimate with \eqref{FirstApproxLogL} and \eqref{ContributionSquaresLogL} and using the basic inequality $(a+b)^{2k}\leq 2^{2k} (|a|^{2k}+|b|^{2k})$, which is valid for all real numbers $a$ and $b$, we deduce that 
\begin{equation}\label{MomentsLogLError}
\begin{aligned}
& \sum_{d\in \mathcal{F}(x)\setminus \mathcal{E}(x)}\left|\log L(s,\chi_D)- \frac{1}{2}\log\left(\frac{1}{s-1/2}\right)\right|^{2k} \\
& \leq e^{O(k)} \left(\sum_{d\in \mathcal{F}(x)} \left|\sum_{p\leq y^2} \frac{\chi_D(p)}{p^s}e^{-p/y}\right|^{2k}+ x\log\log\left(\frac{1}{s-\frac12}\right)^{2k}\right).
\end{aligned}
\end{equation}
Now, it follows from Lemma \ref{LargeSieve} that 
\begin{equation}\label{MomentsLogLError2}
\sum_{d\in \mathcal{F}(x)} \left|\sum_{p\leq y^2} \frac{\chi_D(p)}{p^s}e^{-p/y}\right|^{2k} \ll x\left(k\sum_{p\leq y^2} \frac{e^{-2p/y}}{p^{2s}}\right)^k + x^{5/8} \left(\sum_{p\leq y^2} \frac{e^{-p/y}}{p^s}\right)^{2k}.
\end{equation}
The second term is bounded by 
$$  x^{5/8} \left(\sum_{p\leq y^2} \frac{1}{\sqrt{p}}\right)^{2k} \ll x^{5/8} y^{2k} \ll x^{9/10},$$ 
by our assumption on $k$. Furthermore, the first term of \eqref{MomentsLogLError2} is 
$$ \ll x \left(k \sum_{p} \frac{1}{p^{2s}}\right)^k\ll x\left(2k \log \left(\frac{1}{s-1/2}\right)\right)^{k}.$$
Combining these estimates with \eqref{MomentsLogLError} and \eqref{MomentsLogLError2} we deduce that 
$$\sum_{d\in \mathcal{F}(x)\setminus \mathcal{E}(x)}\left|\log L(s,\chi_D)- \frac{1}{2}\log\left(\frac{1}{s-1/2}\right)\right|^{2k} \ll x\left(Ck \log \left(\frac{1}{s-1/2}\right)\right)^{k}, 
$$
for some positive constant $C>0$. Therefore, the number of fundamental discriminants $D\in \mathcal{F}(x)$ such that 
$$\left|\log L(s,\chi_D)- \frac{1}{2}\log\left(\frac{1}{s-1/2}\right)\right|> V \sqrt{\log \left(\frac{1}{s-1/2}\right)}
$$ 
is 
\begin{align*}
&\ll |\mathcal{E}(x)|+ \left(V^2\log \left(\frac{1}{s-1/2}\right)\right)^{-k}\sum_{d\in \mathcal{F}(x)\setminus \mathcal{E}(x)}\left|\log L(s,\chi_D)- \frac{1}{2}\log\left(\frac{1}{s-1/2}\right)\right|^{2k}\\
&\ll x^{1-(s-1/2)/5} + x \left(\frac{Ck}{V^2}\right)^k.
\end{align*}
Choosing $k= \lfloor V^2/Ce\rfloor$ completes the proof. 
\end{proof}

\section{Oscillations of $L'/L(s, \chi_D)$: Proof of Theorem \ref{Main}}

\subsection{The joint distribution of Dirichlet polynomials}

Let $\{\X(p)\}_{p}$ be a sequence of independent random variables, indexed by the primes, and taking the values $ -1, 0, 1$ with probabilities 
\begin{equation}\label{random_def}
\pr(\X(p)=1)=\pr(\X(p)=-1)= \frac{p}{2(p+1)}, \text{ and } \pr(\X(p)=0)= \frac{1}{p+1}.
\end{equation}
We extend the $\X(p)$ multiplicatively, by defining for $n=\prod_p p^\alpha$, $\X(n) =\prod_p \X(p)^\alpha$. For real numbers $2\leq y< z$ and a complex number $s$ we define
$$ L_{y, z}(s, \chi_D):= \sum_{y<p< z}\frac{\chi_D(p)\log p}{p^{s}},
\quad \text{and} \quad 
 L_{y, z}(s, \X):= \sum_{y<p< z}\frac{\X(p)\log p}{p^{s}}.$$ 

Let $J$ be a positive integer and $s_1, \dots, s_J,u_1, \dots, u_J, v_1, \dots, v_J$ be real numbers such that $1/2<s_j\leq 1$ and $2\leq u_j<v_j\leq x$ for all $1\leq j\leq J$. We also put $\U= (u_1, \dots, u_J)$, $\V=(v_1, \dots, v_J)$ and $\s=(s_1, \dots, s_J)$. 

In this section we shall compare the distribution of the following vector of Dirichlet polynomials 
$$ L_{\U, \V}(\s, \chi_D):=\Big(L_{u_1, v_1}(s_1, \chi_D), L_{u_2, v_2}(s_2, \chi_D), \dots, L_{u_J, v_J}(s_J, \chi_D)\Big), $$ to that of the corresponding probabilistic random vector
$$ L_{\U, \V}(\s, \X):= \Big(L_{u_1, v_1}(s_1, \X), L_{u_2, v_2}(s_2, \X), \dots, L_{u_J, v_J}(s_J, \X)\Big).$$ 
Using the methods of Lamzouri, Lester and Radziwill \cite{LLR2} and \cite{LLR1}  we shall bound the ``discrepancy" between the distribution functions of these vectors, which is defined by
$$ D_{\U, \V}(\s)= \sup_{\mathcal{R}}\left|\frac{1}{|\F|}\left|\big\{ D\in \F : L_{\U, \V}(\s, \chi_D) \in \mathcal{R}\big\}\right|- \pr\big(L_{\U, \V}(\s, \X) \in \mathcal{R}\big) \right| $$ 
where the supremum is taken over all rectangular boxes  (possibly unbounded) $\mathcal{R}\subset \mathbb{R}^J$  with sides parallel to coordinates axes.  

\begin{thm}\label{Discrepancy} Let $x$ be large, and $J\leq (\log \log x)^2$ be a positive integer. Let  $s_1, \dots, s_J$ be real numbers such that $(\log x)^{-1/5}\leq s_j -1/2  \leq (\log\log x)^{-2}$ for all $1\leq j\leq J$. Let $\U, \V\in \mathbb{R}^J$ be such that $\log x\leq u_j<v_j<\exp((\log x)^{1/5})$, and 
$$ (\log u_j) (s_j-1/2) \to 0 \quad \text{ and } \quad (\log v_j) (s_j-1/2) \to \infty, $$ 
as $x\to \infty$, for all $1\leq j\leq J$. 
Then we have

$$D_{\U, \V}(\s)\ll \frac{J}{(\log x)^{1/10}}.$$

\end{thm}

\begin{rem} We did not try to optimize the power of $\log x$ in the above upper bound for the discrepancy $D_{\U, \V}(\s)$. \end{rem}

The first ingredient in the proof of Theorem \ref{Discrepancy} is the following lemma which shows that the moments of $L_{\U, \V}(\s, \chi_D)$ are very close to those of $L_{\U, \V}(\s, \X)$.  


\begin{lem}\label{MomentsDirPoly}
Let $C>0$ be a fixed constant. Let $J$ be a positive integer, and $b_j(n)$ be real numbers such that $|b_j(n)|\leq C $ for all $1\leq j\leq J$ and $ n \geq 1$. 
Let $k_j$ be positive integers for $ j \leq J$  and write $k = \sum_{j\leq J} k_j$.
Then uniformly for $Y, x\geq 2$ we have 
$$
 \frac{1}{|\F|}\sum_{D\in \F}\prod_{ j \leq J}\left(\sum_{n \leq Y} b_j (n) \chi_D(n)\right)^{k_j}  = \ex\bigg[\prod_{ j \leq J}\left(\sum_{n \leq Y} b_j (n)\X(n)\right)^{k_j}\bigg]  + O \bigg(  (CY^{3/2})^{k}x^{-1/2}   \bigg),$$
where the implicit constant in the error term is absolute.
\end{lem}
\begin{proof}
We have
\begin{align*}
  & \frac{1}{|\F|} \sum_{D\in \F}\prod_{ j \leq J}\left(\sum_{n \leq Y} b_j (n) \chi_D(n)\right)^{k_j}   \\
& = \frac{1}{|\F|}\sum_{D\in \F} \bigg( \sum_{n_{i,j} \leq Y}    \prod_{j=1}^J \prod_{i=1}^{k_j} b_j (n_{i,j})\chi_D(n_{i, j})\bigg) \\
& =\sum_{n_{i,j} \leq Y}   \prod_{j=1}^J  \prod_{i=1}^{k_j} b_j (n_{i,j})  \frac{1}{|\F|}\sum_{D\in \F} \chi_D\bigg(\prod_{j=1}^J \prod_{i=1}^{k_j} n_{i,j}\bigg).
\end{align*}


The diagonal terms correspond to those terms such that the product $\prod_{j=1}^J \prod_{i=1}^{k_j} n_{i,j} $ is a perfect square. By Lemma \ref{Orthogonality}, along with the trivial bound $\tau(m)\ll \sqrt{m}$ we find that the contribution of these terms is
\begin{align*}
\Sigma_1& =\sum_{\substack{ n_{i,j}\leq Y \\   \prod n_{i,j}= \square   } }  \prod_{j=1}^J \prod_{i=1}^{k_j} b_j (n_{i,j})  \prod_{p| \prod n_{i,j} }\left(\frac{p}{p+1}\right) + O\left( x^{-1/2} (CY^{3/2})^{k}\right)\\
 & = \ex \bigg(\prod_{ j \leq J}\left(\sum_{n \leq Y} b_j (n)\X(n)\right)^{k_j}   \bigg) +O\left( x^{-1/2} (CY^{3/2})^{k}\right).
\end{align*}
Furthermore, by Lemma \ref{Orthogonality}, the contribution of the off-diagonal terms is 
\begin{align*}
\Sigma_2   = &  \sum_{\substack{ n_{i,j} \leq Y \\   \prod n_{i,j}\neq \square   } }   \prod_{j=1}^J  \prod_{i=1}^{k_j} b_j (n_{i,j})  \frac{1}{|\F|}\sum_{D\in \F} \chi_D\bigg(\prod_{j=1}^J \prod_{i=1}^{k_j} n_{i,j} \bigg)\\
& \ll x^{-1/2} k(CY^{1/4})^{k}\log Y \sum_{\substack{ n_{i,j} \leq Y \\   \prod n_{i,j}\neq \square   } } 1\\
& \ll x^{-1/2} (CY^{3/2})^{k}.
\end{align*}
This completes the proof.
\end{proof}

Let  $\T = (t_1 , \dots, t_J)\in \mathbb R^{J}$, and define 
$$
\Phi_x( \T) = \frac{1}{|\F|}\sum_{D\in \F} 
\exp \left( 2\pi i  \sum_{ j \leq  J } t_jL_{u_j, v_j}(s_j, \chi_D)\right) 
$$
and
\begin{equation}
\begin{aligned}
\Phi_x^{\mathrm{rand}}(\T) = \ex \left(\exp \left( 2\pi i  \sum_{ j \leq  J }t_jL_{u_j, v_j}(s_j, \X)\right) \right).
\end{aligned}\end{equation}

\begin{pro}\label{characteristic}
Let $x$ be large, and $J\leq (\log \log x)^2$ be a positive integer. Let $\U, \V \in \mathbb{R}^J$ be such that $2\leq u_j< v_j<\exp((\log x)^{1/5})$, and  $s_1, \dots, s_J$ be real numbers such that $1/2 \leq s_j \leq 1$ for all $1\leq j\leq J$.
Then, for all $\T$ with $ ||\T||_{\infty} \leq (\log x)^{1/10}$, we have
$$
\Phi_x( \T)  = \Phi_x^{\mathrm{rand}}(\T) + O \left(\exp\left(-(\log x)^{2/3}\right)\right).
$$
Here, $ || \T ||_{\infty} := \sup_{j \leq J} |t_j | $. \end{pro}

\begin{proof}
Let $N=[(\log x)^{2/3}]$.  Using the Taylor expansion of $\exp(2\pi i z)$ for $z\in \mathbb{R}$ we derive
\begin{equation}\label{TaylorExp}
\Phi_x(\T)=\sum_{n=0}^{2N-1} \frac{(2 \pi i)^n}{n!} \cdot \frac{1}{|\F|}\sum_{D\in \F}\left(\sum_{ j \leq  J } t_jL_{u_j, v_j}(s_j, \chi_D)\right)^n+ E_1,
\end{equation}
where 
$$E_1 
\ll  \frac{ ( 2 \pi)^{2N} ||\T||_{\infty}^{2N}}{(2N)!}\frac{1}{|\F|}\sum_{D\in \F}\left(\sum_{j=1}^J|L_{u_j, v_j}(s_j, \chi_D)|\right)^{2N}. 
$$
By Minkowski's inequality and \eqref{LargeSieve2} we obtain
\begin{align*}
\frac{1}{|\F|}\sum_{D\in \F}\left(\sum_{j=1}^J|L_{u_j, v_j}(s_j, \chi_D)|\right)^{2N}
& \leq \left(\sum_{j=1}^J\left(\frac{1}{|\F|}\sum_{D\in \F}|L_{u_j, v_j}(s_j, \chi_D)|^{2N}\right)^{1/(2N)}\right)^{2N}\\
& \ll \big(C_1NJ^2 (\log x)^{2/5}\big)^N.
\end{align*}
Therefore, by our assumption on $||\T||_{\infty}$ and Stirling's formula we deduce that 
\begin{equation}\label{TaylorChar}
\Phi_x(\T)=\sum_{n=0}^{2N-1} \frac{(2 \pi i)^n}{n!} \cdot \frac{1}{|\F|}\sum_{D\in \F}\left(\sum_{ j \leq  J } t_jL_{u_j, v_j}(s_j, \chi_D)\right)^n+ O\left(e^{-N}\right).
\end{equation}

Next, we handle the main term of \eqref{TaylorChar}. To this end, we use Lemma \ref{MomentsDirPoly}, which implies that for all non-negative integers $k_1, k_2, \dots, k_{J}$ such that $k_1+\cdots  + k_{J} \leq 2N$ we have
$$
\frac{1}{|\F|}\sum_{D\in \F}\prod_{ j \leq J}L_{u_j, v_j}(s_j, \chi_D)^{k_j} 
 = \ex\left(\prod_{j=1}^JL_{u_j, v_j}(s_j, \X)^{k_j} \right) + E_2,
$$
where 
$$ E_2\ll x^{-1/2} (C_2 \exp(2(\log x)^{1/5}))^{2N} \ll x^{-1/4}.$$
We use this asymptotic formula for all $0\leq n\leq 2N-1$ to get
\begin{align*}
&\frac{1}{|\F|}\sum_{D\in \F}\left(\sum_{ j \leq  J } t_jL_{u_j, v_j}(s_j, \chi_D)\right)^n \\
&= \sum_{\substack{k_1,\dots, k_{J}\geq 0\\ k_1+\cdots+k_{J}=n}}{n\choose k_1,  \dots, k_{J}}\prod_{j=1}^J t_j^{k_j} \frac{1}{|\F|}\sum_{D\in \F}\prod_{ j \leq J}L_{u_j, v_j}(s_j, \chi_D)^{k_j} \\
&=  \sum_{\substack{k_1,\dots, k_{J}\geq 0\\ k_1+\cdots+k_{J}=n}}{n\choose k_1,  \dots, k_{J}}\prod_{j=1}^J t_j^{k_j} \ex\left(\prod_{j=1}^JL_{u_j, v_j}(s_j, \X)^{k_j} \right)+O\Big(x^{-1/4} \big( \sum_{j=1}^J |t_j| \big)^n\Big)\\
&= \ex\left[ \left(\sum_{ j \leq  J } t_jL_{u_j, v_j}(s_j, \X)\right)^n \right] +O\Big(x^{-1/8}\Big),
\end{align*}
since $\sum_{j=1}^J |t_j| \leq J || \T||_{\infty},$ and $n\leq 2N$. 
Inserting this estimate in \eqref{TaylorChar}, we derive
\begin{equation}\label{EstCharRand1}
\Phi_x(\T)=\sum_{n=0}^{2N-1} \frac{(2 \pi i)^n}{n!} \cdot \ex\left[\left(\sum_{ j \leq  J } t_jL_{u_j, v_j}(s_j, \X)\right)^n\right]+ O \left(\exp\left(-(\log x)^{2/3}\right)\right).
\end{equation}
Furthermore, the proof of Lemma \ref{LargeSieve} implies that for all real numbers $2\leq y<z$, any $\sigma\geq 1/2$ and all positive integers $k$ we have 
\begin{equation}\label{LargeSieve3}
\ex\left[\bigg|\sum_{y\leq p\leq z}\frac{\X(p)\log p}{p^{\sigma}}\bigg|^{2k}\right]\ll \left(C_1 k(\log z)^2\right)^{k},
\end{equation} 
since only the diagonal terms need to be bounded in this case. Using this estimate together with Minkowski's inequality gives 
$$
\ex\left[\left(\sum_{j=1}^J|L_{u_j, v_j}(s_j, \X)|\right)^{2N}\right]
 \ll \big(C_1NJ^2 (\log x)^{2/5}\big)^N.
$$
Thus, by the same argument leading to \eqref{TaylorChar} we deduce that 
$$\Phi_x^{\mathrm{rand}}(\T)= \sum_{n=0}^{2N-1} \frac{(2 \pi i)^n}{n!} \cdot \ex\left[\left(\sum_{ j \leq  J } t_jL_{u_j, v_j}(s_j, \X)\right)^n\right]+ O \left(\exp\left(-(\log x)^{2/3}\right)\right).$$
Combining this estimate with \eqref{EstCharRand1} completes the proof.
\end{proof}

The deduction of Theorem \ref{Discrepancy} from Proposition \ref{characteristic} uses Beurling-Selberg functions.
For $z\in \mathbb C$ let
\[
H(z) =\bigg( \frac{\sin \pi z}{\pi} \bigg)^2 \bigg( \sum_{n=-\infty}^{\infty} \frac{\textup{sgn}(n)}{(z-n)^2}+\frac{2}{z}\bigg)
\qquad\mbox{and} \qquad K(z)=\Big(\frac{\sin \pi z}{\pi z}\Big)^2.
\]
Beurling proved that the function $B^+(x)=H(x)+K(x)$
majorizes $\textup{sgn}(x)$ and its Fourier transform
has restricted support in $(-1,1)$. Similarly, the function $B^-(x)=H(x)-K(x)$ minorizes $\textup{sgn}(x)$ and its Fourier
transform has the same property (see Vaaler \cite[Lemma 5]{Vaaler}).

Let $\Delta>0$ and $a,b$ be real numbers with $a<b$. Take $\mathcal I=[a,b]$
and
define
\[
F_{\mathcal{I}, \Delta} (z)=\frac12 \Big(B^-(\Delta(z-a))+B^-(\Delta(b-z))\Big).
\]
Then we have the following lemma, which is proved in \cite{LLR1} (see  Lemma 7.1  therein and the discussion above it).

\begin{lem} \label{lem:functionbd}
The function $F_{\mathcal{I}, \Delta}$ satisfies the following properties
\begin{itemize}
\item[1.]
For all $x\in \mathbb{R}$ we have $
|F_{\mathcal{I}, \Delta}(x)| \le 1
$ and  
\begin{equation} \label{l1 bd}
0 \le \mathbf 1_{\mathcal I}(x)- F_{\mathcal{I}, \Delta}(x)\le K(\Delta(x-a))+K(\Delta(b-x)).
\end{equation}
\item[2.]
 The Fourier transform of $F_{\mathcal{I}, \Delta}$ is
\begin{equation} \label{Fourier}
\widehat F_{\mathcal{I}, \Delta}(y)=
\begin{cases}\widehat{ \mathbf 1}_{\mathcal I}(y)+O\Big(\frac{1}{\Delta} \Big) \mbox{ if } |y| < \Delta, \\
0 \mbox{ if } |y|\ge \Delta.
\end{cases}
\end{equation}
\end{itemize}
\end{lem}
Before proving Theorem \ref{Discrepancy} we require the following lemma, which gives an asymptotic for the variance of $L_{u, v}(s, \X)$, and shows that its characteristic function decays like the Gaussian in a certain range, assuming that $(\log u)(s-1/2)\to 0$ and $(\log v)(s-1/2)\to \infty$.
\begin{lem}\label{VarianceCharacteristic} Let $x$ be large, and $s=1/2+ 1/g(x)$ where $g(x)\to \infty$ as $x\to \infty$. Let $u$ and $v$ be such that 
$$ \frac{\log u}{g(x)} \to 0 \ \text{ and } \  \frac{\log v}{g(x)} \to \infty \text{ as } x\to \infty.$$ 
Then we have 
\begin{equation}\label{VarianceDir}
\ex\left(|L_{u, v} (s, \X)|^2\right)=\sum_{u< p< v} \frac{(\log p)^2p}{p^{2s}(p+1)}= \frac{g(x)^2}{4}(1+o(1)).
\end{equation}
Moreover, for all real numbers $\xi$ such that $|\xi| \leq u^{1/3}$ we have 
\begin{equation} \label{BoundChar} \ex\big(\exp\left( i \xi L_{u, v}(s, \X)\right)\big) \ll \exp\left(- \frac{\xi^2}{10} g(x)^2\right).
\end{equation}
\end{lem}
\begin{proof}
We start by establishing \eqref{VarianceDir}. By the independence of the $\X(p)$'s we have 
$$ \ex\left(|L_{u, v} (s, \X)|^2\right)= \sum_{u<p<v} \frac{\ex(|\X(p)|^2)(\log p)^2}{p^{2s}}= \sum_{u< p< v} \frac{(\log p)^2p}{p^{2s}(p+1)}. $$
Now, note that 
$$ \sum_{p\leq u} \frac{(\log p)^2p}{p^{2s}(p+1)} \leq \sum_{p\leq u} \frac{(\log p)^2}{p} = (1/2+o(1)) (\log u)^2 = o\big( g(x)^2\big).
$$
Let $V= (\log v)/g(x).$ By the prime number theorem and partial summation we obtain
$$  \sum_{v\leq p} \frac{(\log p)^2 p}{p^{2s} (p+1)} \ll \frac{\log v}{v^{2s-1}(s-1/2)} + \frac{1}{(s-1/2)^2 v^{2s-1}}
\ll  g(x)^2 Ve^{-2V}=  o\big( g(x)^2\big).
$$
Therefore, \eqref{VarianceDir} follows upon noting that 
$$\sum_{p} \frac{(\log p)^2p}{p^{2s}(p+1)}= \frac{(1+o(1))}{(2s-1)^2}= \frac{g(x)^2}{4}(1+o(1)), 
$$ 
by partial summation and the prime number theorem. 

We now establish \eqref{BoundChar}. By the independence of the $\X(p)$'s we have 
$$  \ex\big(\exp\left(i \xi L_{u, v}(s, \X)\right)\big)= \prod_{u<p<v} \ex\left(\exp\left(i \xi \frac{\X(p) \log p}{p^s} \right)\right).$$
Note that $|\xi (\log p)/p^s| <1$ for all primes $p>u$ by our assumption on $\xi$ and since $s\geq 1/2$. Hence for all primes $u<p<v$ we have
\begin{align*}
\ex\left(\exp\left(i \xi \frac{\X(p) \log p}{p^s} \right)\right) &=  \ex\left(1+ i\xi \frac{\X(p) \log p}{p^s} - \frac{\xi^2}{2} \frac{\X(p)^2 (\log p)^2}{p^{2s}}  + O\left(\frac{|\xi|^3(\log p)^3}{p^{3s}}\right) \right) 
\\
& = 1- \frac{\xi^2}{2} \frac{(\log p)^2 p}{p^{2s}(p+1)} +O\left(\frac{|\xi|^3(\log p)^3}{p^{3/2}}\right),
\end{align*}
since $\ex(\X(p))=0$ and $\ex(\X(p)^2)=p/(p+1).$ The bound \eqref{BoundChar} follows upon noting that 
$\sum_{u<p<v}((\log p)^2 p)/(p^{2s}(p+1))\sim g(x)^2/4$ by \eqref{VarianceDir}, together with the fact that $\sum_{p>u} (\log p)^3/p^{3/2}\ll (\log u)^3/\sqrt{u}. $
\end{proof}

\begin{proof}[Proof of Theorem \ref{Discrepancy}]
First, we show that it suffices to consider rectangular regions $R$ contained in $[-\sqrt{\log x}, \sqrt{\log x}]^{J}.$ Indeed, let $R$ be a rectangular box in $\mathbb{R}^J$ and put $\widetilde{R}=  R \cap [-\sqrt{\log x}, \sqrt{\log x}]^{J} $. Let $1\leq j\leq J$ and take $k= [\sqrt{\log x}]$. Then, it follows from \eqref{LargeSieve2} that the number of fundamental discriminants $|D|\leq x$ such that $|L_{u_j, v_j}(s_j, \chi_D)|>\sqrt{\log x}$ is 
\begin{align*}
&\ll \frac{1}{(\sqrt{\log x})^{2k}}\sum_{D\in \F}|L_{u_j, v_j}(s_j, \chi_D)|^{2k}\ll x\left(C_1 \frac{k (\log v_j)^2}{\log x}\right)^k\ll xe^{-k}\ll x\exp\left(-\sqrt{\log x}\right).
\end{align*}
Therefore, we obtain 
$$ \frac{1}{|\F|}\left|\left\{d\in \F : L_{\U, \V}(\s, \chi_D) \notin  [ -\sqrt{\log x}, \sqrt{\log x}]^{J}\right\}\right| \ll J \exp\left(-\sqrt{\log x}\right),$$
and hence 
\begin{align*}
& \frac{1}{|\F|}\left|\left\{d\in \F : L_{\U, \V}(\s, \chi_D) \in R \right\}\right| - \frac{1}{|\F|}\left|\left\{d\in \F : L_{\U, \V}(\s, \chi_D) \in \widetilde{R} \right\}\right|\\
& \ll J \exp\left(-\sqrt{\log x}\right).
\end{align*}
A similar argument using  \eqref{LargeSieve3} instead of \eqref{LargeSieve2}  shows that
$$ \pr \big(L_{\U, \V} ( \s, \X) \in  R  \big)- \pr \big(L_{\U, \V} ( \s, \X) \in  \widetilde{R} \big)
 \ll J \exp\left(-\sqrt{\log x}\right).$$
Therefore, we might now suppose that $R$  is contained in $[-\sqrt{\log x}, \sqrt{\log x}]^{J}.$ 

Let 
$\Delta:=(\log x)^{1/10}$,
 and $ R=\prod_{j=1}^{J}\mathcal I_j $ for $j=1,\ldots, J$, with $\mathcal I_j=[a_j,b_j]$ and 
$0< b_j-a_j \le 2\sqrt{\log x}$. By Fourier inversion, \eqref{Fourier}, and Proposition \ref{characteristic} we have that
\begin{equation} \label{long est}
\begin{aligned}
& \frac{1}{|\F|}  \sum_{D\in \F}\prod_{j=1}^J F_{\mathcal I_j, \Delta} \Big(L_{u_j, v_j}(s_j, \chi_D)\Big) 
 =\int_{\mathbb R^{J}} \prod_{j=1}^J \widehat{F}_{\mathcal I_j, \Delta} (t_j) 
  \Phi_x(\T) \, d\T   \\
&
= \int\limits_{\substack{|t_j| \le \Delta \\ j=1,2, \ldots, J}}  \prod_{j=1}^J \widehat{F}_{\mathcal I_j, \Delta} (t_j) 
  \Phi_x^{\mathrm{rand}}(\T) \, d\T  +O\left(\frac{\big(2\Delta \sqrt{\log x}\big)^{J}}{\exp\left((\log x)^{2/3}\right)}\right)\\
& 
=\mathbb E \left( \prod_{j=1}^J F_{\mathcal I_j, \Delta} \Big(L_{u_j, v_j}(s_j, \X)\Big)\right) + O\left(\frac{\big(2\Delta \sqrt{\log x}\big)^{J}}{\exp\left((\log x)^{2/3}\right)}\right).
\end{aligned}
\end{equation}

Next note that $\widehat K(\xi)=\max(0,1-|\xi|)$. Applying Fourier inversion, Proposition \ref{characteristic} with $J=1$,
and \eqref{BoundChar}  we obtain
\begin{align*}
& \frac{1}{|\F|}  \sum_{D\in \F} K\Big( \Delta \cdot \Big( L_{u_1, v_1}(s_1, \chi_D)-\alpha\Big)\Big) \\
&=\frac{1}{\Delta}\int_{-\Delta}^{\Delta}\Big(1-\frac{|\xi|}{\Delta}\Big) e^{-2\pi i \alpha \xi} \Phi_x (\xi) \, d\xi \ll  \frac{1}{\Delta},
\end{align*}
where $\alpha$ is an arbitrary real number. By this and \eqref{l1 bd} we get that
\begin{equation} \label{K bd}
  \sum_{D\in \F}  F_{\mathcal I_1, \Delta}\Big(L_{u_1, v_1}(s_1, \chi_D)\Big)
=  \sum_{D\in \F} \mathbf 1_{\mathcal I_1}\Big(L_{u_1, v_1}(s_1, \chi_D)\Big)+O( |\F|/\Delta).
\end{equation}
Lemma \ref{lem:functionbd} implies that $F_{\mathcal I_j, \Delta}(x) \leq \mathbf 1_{\mathcal I_j}(x)$ and $|F_{\mathcal I_j, \Delta}(x)| \le 1$ for all  $j=1,\ldots, J$. Hence, using these facts and \eqref{K bd} we obtain
\begin{align*}
& \frac{1}{|\F|}  \sum_{D\in \F}\prod_{j=1}^J F_{\mathcal I_j, \Delta} \Big(L_{u_j, v_j}(s_j, \chi_D)\Big)\\
&=\frac{1}{|\F|}  \sum_{D\in \F} \mathbf 1_{\mathcal I_1} \Big( L_{u_1, v_1}(s_1, \chi_D)\Big) \times \prod_{j=2}^J F_{\mathcal I_j, \Delta} \Big(L_{u_j, v_j}(s_j, \chi_D)\Big) +O(1/\Delta).
\end{align*}
Iterating this argument and using the analogs of \eqref{K bd} for $ L_{u_j, v_j}(s_j, \chi_D)$ with $2\leq j\leq J$, which
are proved similarly, we derive
\begin{equation} \label{one}
\begin{split}
 \frac{1}{|\F|}  \sum_{D\in \F}\prod_{j=1}^J F_{\mathcal I_j, \Delta} \Big(L_{u_j, v_j}(s_j, \chi_D)\Big) 
& = \frac{1}{|\F|}  \sum_{D\in \F} \prod_{j=1}^J \mathbf 1_{\mathcal I_j} \Big(L_{u_j, v_j}(s_j, \chi_D)\Big) 
 +O\left(\frac{J}{\Delta}\right) \\
=&   \frac{1}{|\F|}\left|\big\{ D\in \F : L_{\U, \V}(\s, \chi_D) \in \mathcal{R}\big\}\right|+O\left(\frac{J}{\Delta}\right).
\end{split}
\end{equation}
A similar argument shows that 
\begin{equation} \label{two}
\begin{aligned}
\mathbb E \bigg( \prod_{j=1}^J F_{\mathcal I_j, \Delta} \Big(L_{u_j, v_j}(s_j, \X)\Big) \bigg) = \pr\big(L_{\U, \V}(\s, \X) \in \mathcal{R}\big) +O\left(\frac{J}{\Delta}\right).
\end{aligned}
\end{equation}
Inserting the estimates \eqref{one} and \eqref{two} in \eqref{long est} completes the proof.

\end{proof}

\subsection{Proof of Theorem \ref{Main}}\label{SectionProofThmMain}

 Let $S^{-}(a_1,a_2, \dots , a_k)$ denote the number of sign changes in the sequence $a_1, a_2, \dots, a_k$ with zero terms deleted, and let $S^{+}(a_1,a_2, , \dots, a_k)$  denote the maximum number of sign changes with zero terms replaced by a number of arbitrary sign. We will use the following lemma from \cite{BaMo}.
\begin{lem}[Lemma 6 of \cite{BaMo}]\label{lem:BaMo}
Let $\delta>0$ and suppose that $Z_1, Z_2 \dots, Z_R$ are independent random variables such that 
$\pr(Z_j>0)\geq \delta$ and $\pr(Z_j<0)\geq \delta$ for all $j$. Then 
$$ \pr\left(S^-(Z_1, \dots, Z_R)\leq \frac{\delta R}{5}\right)\ll e^{-\delta R/3}, 
$$
uniformly in $\delta$ and $R$.
\end{lem}
To prove Theorem \ref{Main} we shall closely follow  the lines of the proof of Baker and Montgomery \cite{BaMo}, incorporating our ingredients developed in Proposition \ref{TruncationPrimes} and Theorem \ref{Discrepancy}, and using a somewhat more suitable choice of sampled points. In particular, Theorem \ref{Main} is a direct consequence of the following technical result, where we show that $\frac{L'}{L}(s, \chi_D)$ has many sign changes with size $1/(s-1/2)$, as $s\to 1/2$.
\begin{thm}\label{Thm.SignChangesL'L}
Let $0<\alpha<1/20$ be a small fixed constant. Let $x$ be large and $100\leq M\leq \log_2 x$. Define
$$
R:= \left\lfloor \frac{\alpha \log_2 x}{\log M} \right\rfloor,
$$ 
and for $R/5\leq r\leq R$ consider the points 
$$ s_r:= \frac{1}{2} + \frac{1}{M^{3r}}.$$
For each $D\in \F$ and $R/5\leq r\leq R$ we let $Z_r(D)\in \{-1, 0, 1\}$, such that $Z_r(D)=1$ if $-\frac{L'}{L}(s_r, \chi_D)>M^{3r}$, $Z_r(D)=-1$ if $-\frac{L'}{L}(s_r, \chi_D)<-M^{3r}$, and $Z_r(D)=0$ otherwise. Let $N_R(D)$ be the number of sign changes in the sequence $\left\{Z_r(D)\right\}_{R/5\leq r\leq R}$ with zero terms deleted. There exists an absolute constant $\delta>0$ such that for all $D\in \F$ we have 
$$ N_R(D)> \frac{\delta R}{5}, $$
except for a set of discriminants of size 
$$ \ll x\exp\left(-\frac{\delta R}{100}\right) + x R \exp\left(-\frac{M^2}{60}\right). $$

\end{thm}

\begin{proof}  First, we observe that $s_r$ is strictly decreasing and  $(\log x)^{-1/5}< s_r-1/2<(\log\log x)^{-2}$ for all $R/5\leq r\leq R$. 
Let us also define 
$$ u_r:= \exp\left(M^{3r-1}\right) \text{ and } v_r:= \exp\left(M^{3r+1}\right).
$$
Then we note that $v_r <u_{r+1} $ and $\log x <u_r <v_r <\exp((\log x)^{1/5})<$ for all $R/5\leq r\leq R$. Hence, $R$, $\U, \V$ and $\s$ verify the assumptions of Proposition \ref{TruncationPrimes} and  Theorem \ref{Discrepancy}. 
Let $\mathcal{E}_4(x)$ be the set of fundamental discriminants $|D|\leq x$ such that 
$$ 
\left|\frac{L'}{L}(s_r,\chi_D)+ \sum_{u_r<  p< v_r} \frac{\chi_D(p)\log p}{p^{s_r}} \right| > M^{3r}, 
$$
for some $R/5\leq r\leq R$. Then, it follows from Proposition \ref{TruncationPrimes} that 
\begin{equation}\label{exceptionalset}
|\mathcal{E}_4(x)| \ll x R \exp\left(-\frac{M^2}{60}\right).
\end{equation}
From the central limit theorem, applied to the random variables defined in \eqref{random_def} we conclude that
\begin{equation}\label{Eq.BoundDelta}
    \begin{aligned}
\pr\left(L_{u_r, v_r}(s_r, \X)> 2M^{3r}\right)&= \pr\left(L_{u_r, v_r}(s_r, \X)<- 2M^{3r}\right)\\
& =(1+o(1))\frac{1}{\sqrt{2\pi}}\int_{4}^{\infty} e^{-x^2/2}dx,
\end{aligned}
\end{equation}
  since by Lemma \ref{VarianceCharacteristic} the variance of the sum of random variables $L_{u_r, v_r}(s_r, \X)=\sum_{u_r< p< v_r} \frac{\X(p)\log p}{p^{s_r}}$ is 
$$
\sum_{u_r< p< v_r} \frac{(\log p)^2p}{p^{2s_r}(p+1)}= \frac{(1+o(1))}{(2s_r-1)^2}=  \left(\frac{1}{4}+o(1)\right)M^{6r}.
$$ We observe that the intervals $(u_r,v_r)$ are disjoint making the variables $L_{u_r, v_r}(s_r, \X)$  independent. Define the random variable $Y_r$ by $Y_r=1$ if $L_{u_r, v_r}(s_r, \X) > 2M^{3r}$, $Y_r=-1$ if $L_{u_r, v_r}(s_r, \X)< - 2M^{3r}$ and $Y_r=0$ otherwise. Let $N_R$ be the number of sign changes in the sequence $\left\{Y_r\right\}_{R/5\leq r\leq R}$ with zero terms deleted. By \eqref{Eq.BoundDelta} and Lemma \ref{lem:BaMo} we deduce that 
$$ \pr(N_R \leq \delta R) \ll \exp\left(-\frac{\delta R}{100}\right),$$
where $\delta:=\frac{1}{\sqrt{2\pi}}\int_{5}^{\infty} e^{-x^2/2}dx.$ 

Let us now define the vector $L_{\U, \V}(\s, \chi_D):= (L_{u_r, v_r}(s_r, \chi_D))_{R/5\leq r \leq R}$ and  the corresponding random vector
$L_{\U, \V}(\s, \X):=(L_{u_r, v_r}(s_r, \X))_{R/5\leq r \leq R}.$ The event $\left\{N_R > \delta R \right\}$ can be decomposed as a disjoint union of events $\displaystyle{\bigcup_{j=1}^{J} \left\{L_{\U, \V}(\s, \X) \in \mathcal{R}_j\right\}}$ where the $\mathcal{R}_j$'s  are (possibly) unbounded rectangular boxes and $J \leq 3^R$. Hence 
$$
 \sum_{j=1}^{J}\pr\big(L_{\U, \V}(\s, \X) \in \mathcal{R}_j\big)= \pr(N_R > \delta R) =1 + O\left(\exp\left(-\frac{\delta R}{100}\right)\right).  
$$ 
Applying Theorem \ref{Discrepancy}, upon noting that $\frac{R3^R}{(\log x)^{1/10}} \ll \exp\left(-\frac{\delta R}{100}\right),$ we get  
$$ 
\frac{1}{|\F|}\Big|\Big\{ D\in \F : L_{\U, \V}(\s, \chi_D) \in \bigcup_{j=1}^J\mathcal{R}_j\Big\}\Big| = 1 + O\left(\exp\left(-\frac{\delta R}{100}\right) \right).
$$  Similarly as above we can define the random variable $Y_{r,D}$ by $Y_{r,D}=1$ if $L_{u_r, v_r}(s_r, \chi_D) > 2M^{3r}$, $Y_{r,D}=-1$ if $L_{u_r, v_r}(s_r, \chi_D)< - 2M^{3r}$ and $Y_{r,D}=0$ otherwise, and we let $\widetilde{N}_{R}(D)$ denote the number of sign changes in the sequence $\left\{Y_{r,D}\right\}_{R/5\leq r\leq R}$ with zero terms deleted. Then we deduce from the previous estimate that 
$$\frac{1}{|\F|}\left|\big\{ D\in \F : \widetilde{N}_R(D)> \delta R \big\}\right| = 1 + O\left(\exp\left(-\frac{\delta R}{100}\right)\right).$$ 
Combining this estimate with \eqref{exceptionalset} completes the proof.
\end{proof}
\begin{proof}[Proof of Theorem \ref{Main}]
It follows from Theorem \ref{Thm.SignChangesL'L} that for all $D\in \F$ except for a set $\mathcal{E}_5(x)$ the vector $\left(\frac{L'}{L}(s_r,\chi_D)\right)_{R/5\leq r \leq R}$ has at least $\gg (\log_2 x)/\log M$ sign changes, and moreover we have 
$$ |\mathcal{E}_5(x)| \ll  x\exp\left(-\frac{\delta R}{100}\right) + x R \exp\left(-\frac{M^2}{60}\right).$$
The result follows upon choosing $M= (\log_3 x)^A$ in part (1), and $M=\log_2 x $ in part (2) of the theorem.
\end{proof}

\section{Sign changes of quadratic character sums: Proof of Theorem \ref{ThmPartialSumsPositive}}

Recall that for $\Re(s)>0$ we have 
\begin{equation}\label{Eq.LaplaceCharacters} \frac{L(s, \chi_D)}{s} \left(-\frac{L'}{L}(s, \chi_D)+\frac{1}{s}\right)
= \int_0^{\infty} t S_{\chi_D}(e^t) e^{-st}dt.
\end{equation}

In order to prove Theorem \ref{ThmPartialSumsPositive} we shall truncate the Laplace transform on the right hand side of this identity, in order to detect sign changes of $S_{\chi_D}(t)$ in short initial intervals.
\begin{pro}\label{pro.TruncationLaplaceChar}
    Let $\ep>0$ be small and fixed. Let $x$ be large and $s=\frac12+\frac1K$ with  $(\log x)^{\ep} \le  K \leq (\log x)/(\log\log x)^2.$ Let $0\leq Y\leq K^{1/4}$ and $Z\geq K(\log K)^2$ be real numbers. Then for all but $O(x/Y)$ fundamental discriminants $D\in \F$ we have
    $$\left|\int_0^{\infty} t S_{\chi_D}(e^t) e^{-st}dt- \int_{Y}^{Z} t S_{\chi_D}(e^t) e^{-st}dt\right|\leq  Y^4.$$
\end{pro}
\begin{proof}
We will proceed by bounding the moments
$$\mathcal{I}_1:=\sum_{D\in \F} \left|\int_{Z}^{\infty} t S_{\chi_D}(e^t) e^{-st}dt\right|^2,$$
and 
$$\mathcal{I}_2:= \sum_{D\in \F} \left|\int_{0}^{Y} t S_{\chi_D}(e^t) e^{-st}dt\right|^2.$$
 We start with $\mathcal{I}_1$. By the Cauchy-Schwarz inequality we have
\begin{equation}\label{Eq.CauchyChar1}
\begin{aligned}
    \left|\int_{Z}^{\infty} t S_{\chi_D}(e^t) e^{-st}dt\right|^2 &\leq \left(\int_{Z}^{\infty} t^2e^{-t/K}dt\right)\left(\int_{Z}^{\infty} |S_{\chi_D}(e^t)|^2e^{-t/K}e^{-t}dt\right)\\
& \ll Z^{-1/(2K)} \int_{Z}^{\infty} |S_{\chi_D}(e^t)|^2e^{-t/K}e^{-t}dt.
\end{aligned}
\end{equation}
Therefore we obtain
\begin{equation}\label{Eq.TruncCharSum1}
\mathcal{I}_1
\ll Z^{-1/(2K)} \int_{Z}^{\infty} \left(\sum_{D\in \F}|S_{\chi_D}(e^t)|^2\right)e^{-t/K}e^{-t}dt.
\end{equation}
Now it follows from Lemma \ref{lemmaJutila} that for all $t\ge 0$ we have
$$ \sum_{D\in \F}|S_{\chi_D}(e^t)|^2 = \sum_{D\in \F}\Big|\sum_{n\leq e^t} \chi_D(n)\Big|^2 \ll xe^{t} (\log x)^8.
$$
Inserting this estimate in \eqref{Eq.TruncCharSum1} gives
\begin{equation}\label{Eq.BoundI1CharSum}
\mathcal{I}_1 \ll x(\log x)^8 Z^{-1/(2K)} \int_{Z}^{\infty} e^{-t/K}dt\ll x,
\end{equation}
by our assumption on $K$. 

We now bound $\mathcal{I}_2$. 
Note that for $0\leq t\leq Y$  we have $e^{t/K}=1+o(1)$ and hence by the Cauchy-Schwarz inequality we have
\begin{equation}\label{Eq.CauchyChar2}
\begin{aligned}
    \left|\int_{0}^{Y} t S_{\chi_D}(e^t) e^{-st}dt\right|^2 &\ll \left(\int_{0}^{Y} t^2dt\right)\left(\int_{0}^{Y} |S_{\chi_D}(e^t)|^2e^{-t}dt\right)\\
& \ll Y^3 \int_{0}^{Y} |S_{\chi_D}(e^t)|^2e^{-t}dt.
\end{aligned}
\end{equation}
In this range, we shall use the large sieve inequality \eqref{Eq.BabyLargeSieve} instead of Lemma \ref{lemmaJutila}. This gives 
$$ \sum_{D\in \F}|S_{\chi_D}(e^t)|^2 \ll (x+ e^{2t}t) \sum_{\substack{m, n \leq e^t\\ mn= \square}} 1 \ll x\sum_{j\leq e^t} \tau(j^2)\ll x t^3e^t,$$
by standard estimates on divisor functions\footnote{This follows from example from the Selberg-Delange method, since $\tau(p^2)=3$ for every prime $p$.}. Combining this bound with \eqref{Eq.CauchyChar2} we derive 
$$
    \mathcal{I}_2  \ll Y^3 \int_0^{Y} \left(\sum_{D\in \F}|S_{\chi_D}(e^t)|^2\right) e^{-t} dt\ll xY^3\int_{0}^{Y} t^3dt \ll x Y^7.
$$
Using this estimate together with \eqref{Eq.BoundI1CharSum} we deduce that the number of fundamental discriminants $D\in \F$ such that 
$$ \left|\int_0^{\infty} t S_{\chi_D}(e^t) e^{-st}dt- \int_{Y}^{Z} t S_{\chi_D}(e^t) e^{-st}dt\right|> Y^4$$
is 
\begin{equation}\label{Eq.DeduceTruncationChar}
    \ll Y^{-8} (\mathcal{I}_1+ \mathcal{I}_2) \ll \frac xY,
\end{equation}
as desired.
\end{proof}

In order to detect sign changes of the character sum $S_{\chi_D}$ we shall use a standard result in analysis relating the sign changes of a function to sign changes of its Laplace transform. For a real-valued function $g$ defined on an interval $(a, b)$ we let $S^{\pm}(g; a, b)$ be the supremum of $S^{\pm}(g(a_1), \dots, g(a_k))$ over all finite sequences for which $a < a_1 < a_2 < ... < a_k < b$ (where $S^{\pm}(g(a_1), \dots, g(a_k))$ was defined in Section \ref{SectionProofThmMain}). Then we have the following result.  
\begin{lem}[Lemma 7 of \cite{BaMo}]\label{lem:Karlin}
Let $g$ be a real-valued function defined on $\mathbb{R}$ which is Riemann-integrable on finite intervals, and suppose that the Laplace transform
$$\mathcal{L}(s):=\int_{-\infty}^{\infty} g(x) e^{-sx} dx$$
converges for all $s>0$. Then
$$S^-(g;-\infty, \infty)\geq S^{+}(\mathcal{L}; 0, \infty).$$
\end{lem}

\begin{proof}[Proof of Theorem \ref{ThmPartialSumsPositive}]
We will use Theorem \ref{Thm.SignChangesL'L} with $M=(\log_3x)^A$. As in this result we define
$$
R= \left\lfloor \frac{\alpha \log_2 x}{\log M} \right\rfloor
$$ 
and for $R/5\leq r\leq R$ we put
$$ s_r= \frac{1}{2} + \frac{1}{M^{3r}},$$
and let $K_r=M^{3r}=1/(s_r-1/2).$ Note that $(\log x)^{2\alpha/5} \leq K_r\leq (\log x)^{3\alpha}$ for all $R/5\leq r\leq R$. We also let 
$$ Y_1:= (\log x)^{\alpha/20} \ \text{ and } \ Y_2:= (\log x)^{4\alpha}.$$
Let $\delta$ be the constant in Theorem \ref{Thm.SignChangesL'L} and define $\mathcal{H}(x)$ to be the set of fundamental discriminants $D\in \F$ such that the following conditions hold:
\begin{itemize}
    \item [1.] We have $N_R(D)>\delta R/5$, where $N_R(D)$ is defined in Theorem \ref{Thm.SignChangesL'L}. 
    \item[2.] For all $R/5\leq r\leq R$ we have 
    $$\left|\log L(s_r, \chi_D)- \frac{\log K_r}{2}\right|\leq \frac{\log K_r}{4}.$$
    \item[3.] For all $R/5\leq r\leq R$ we have
    $$ \left|\int_0^{\infty} t S_{\chi_D}(e^t) e^{-s_rt}dt- \int_{Y_1}^{Y_2} t S_{\chi_D}(e^t) e^{-s_rt}dt\right|\leq  Y_1^4.$$
\end{itemize}
Then it follows from Propositions \ref{CLTLogL} and \ref{pro.TruncationLaplaceChar} and Theorem \ref{Thm.SignChangesL'L}, that 
\begin{equation}\label{Eq.OrderHx}
|\F\setminus \mathcal{H}(x)| \ll x  \exp\left(-(\log_3 x)^A\right).
\end{equation}
Let $Z_r(D)$ be as in Theorem \ref{Thm.SignChangesL'L}. Then it follows from our conditions above together with the identity \eqref{Eq.LaplaceCharacters} that for every $D\in \mathcal{H}(x)$ and all $R/5\leq r\leq R$ we have 
$$ \int_{Y_1}^{Y_2}t S_{\chi_D}(e^t) e^{-s_rt}dt> K_r,$$
if $Z_r(D)=1$,  while if $Z_r(D)=-1$ then we have 
$$ \int_{Y_1}^{Y_2}t S_{\chi_D}(e^t) e^{-s_rt}dt<- K_r.$$
Therefore, for all $D\in \mathcal{H}(x)$ we have $$S^{-}\left(\mathcal{L}_D(s), 0, \infty\right)>\frac{\delta}{5} R,$$
where 
$$\mathcal{L}_D(s)= \int_{Y_1}^{Y_2}t S_{\chi_D}(e^t) e^{-st}dt, $$
is the Laplace transform of the function
$$
g(t)=\begin{cases} t S_{\chi_D}(e^t) &\text{ if } Y_1\leq t\leq Y_2,\\
 0 &\text{ otherwise. }\end{cases}
 $$
Appealing to Lemma \ref{lem:Karlin} completes the proof.

\end{proof}


\section{Lower bounds on the number of real zeros of Fekete polynomials and the associated Theta functions}

\subsection{Real zeros of Fekete polynomials}

The proof of Theorem \ref{MainFek} proceeds similarly to that of Theorem \ref{ThmPartialSumsPositive}, but there are additional technical difficulties in this case. Indeed, although one can think of $F_D(e^{-t})$ as being close to a character sum of length approximately $1/t$, we do not have a result like Jutila's estimate \eqref{lem.Jutila} in this case, which is uniform in $t$. Instead, if $D^{-1}<t<D^{-1/2}$ (so $F_D(e^{-t})$ behaves like a character sum of length $>\sqrt{D}$) we shall use the Poisson summation formula to relate $F_D(e^{-t})$ to a dual character sum having length approximately $Dt<\sqrt{D}$. We prove the following lemma.
\begin{lem}\label{truncatedl}
Let $D \in \F$ be a positive discriminant and $1\leq T\leq D$ be a real number. Then we have 
$$  F_D\left(\exp\left(-\frac{T}{D}\right)\right) = \frac{2\sqrt{D}}{T}\sum_{n\geq 1} \frac{\chi_D(n)}{1+4\pi^2(n/T)^2}+ O\left(\frac{D}{T}e^{-T}\right).
$$
\end{lem}

We deduce this result from the following classical Poisson summation-type formula.
\begin{lem}\label{lem:Poisson}
\cite[Theorem $10.5$]{Koukoul}.
Let $f$ be a piecewise differentiable function such that $\hat{f}(x) \ll \frac{1}{x^2}$, where $\hat{f}$ denotes the Fourier transform of $f$. For any primitive character modulo an integer $q\geq 2$ and $N >0$, we have 
$$ \sum_{n\in \mathbb{Z}} \chi(n)f\left(\frac{n}{N}\right) = \frac{\chi(-1)N}{\tau(\overline{\chi})} \sum_{n\in \mathbb{Z}} \overline{\chi}(n)\hat{f}\left(\frac{nN}{q}\right),
$$
where $\tau(\chi)$ denotes the usual Gauss sum associated to the primitive character $\chi$.
\end{lem}
\begin{proof}[Proof of Lemma \ref{truncatedl}] Since $D$ is positive, $\chi_D$  is an even character. Hence
\begin{align*}
    F_D\left(\exp\left(-\frac{T}{D}\right)\right) &= \frac{1}{2}\sum_{1 \leq \vert n\vert  \leq  D-1} \chi_D(n)\exp\Big(-\frac{\vert n\vert T}{D}\Big)\\
    &= \frac{1}{2}\sum_{n\in \mathbb{Z}} \chi_D(n)\exp\Big(-\frac{\vert n\vert T}{D}\Big)+ O\left(\frac{D}{T}e^{-T}\right), 
\end{align*}
since 
$$\sum_{n\geq D} \exp\Big(-\frac{n T}{D}\Big) \ll \frac{D}{T}e^{-T}.$$
Note that $\tau(\chi_D)=\sqrt{D}$ since $D$ is a positive discriminant. Applying Lemma \ref{lem:Poisson} with $f(t)=e^{-|t|}$ and using that $\xi \to 2/(1+(2\pi \xi)^2)$ is the Fourier transform of $f$ completes the proof.
 \end{proof}
Using Lemma \ref{truncatedl} and proceeding as in the proof of Proposition \ref{pro.TruncationLaplaceChar} we establish the following result, which allows us to deduce Theorem \ref{MainFek} from Theorem \ref{Thm.SignChangesL'L}.
\begin{pro}\label{pro.TruncationLaplaceFek}
Let $\ep>0$ be small and fixed. Let $x$ be large and $s=\frac12+\frac1K$ with 
$ (\log x)^{\ep} \le  K \leq (\log x)/(\log\log x)^3.$ Let $1\leq Y\leq \exp(K^{1/5})$ and $Z\geq  \exp\left(K(\log K)^2\right)$. Then for all but $O(x/\log Y)$ fundamental discriminants $0<D\leq x$ we have
$$ \left|\int_{0}^{\infty}\frac{F_D(e^{-t})}{1-e^{-Dt}}t^{s-1}(\log t) dt-\int_{Z^{-1}}^{Y^{-1}} \frac{F_D(e^{-t})}{1-e^{-Dt}}t^{s-1}(\log t) dt\right|\leq (\log Y)^5.$$
\end{pro}

 \begin{proof}[Proof of Proposition \ref{pro.TruncationLaplaceFek}]
 First, observe that we can restrict our attention to positive fundamental discriminants $D$ in the range $x/(\log x)^2\leq D\leq x$, since the number of those in the range $0< D<x/(\log x)^2$ is $\ll x/(\log x)^2\ll x/\log Y $. Let $\G(x)$ be the set of such discriminants. Similarly to the proof of Proposition \ref{pro.TruncationLaplaceChar} we will bound the second moments
$$
    \mathcal{J}_1:=\sum_{D\in \G(x)} \left|\int_0^{Z^{-1}} F_D(e^{-t})(1-e^{-Dt})^{-1}t^{s-1} (\log t)dt\right|^2, 
$$
and 
$$
    \mathcal{J}_2:=\sum_{D\in \G(x)} \left|\int_{Y^{-1}}^{\infty} F_D(e^{-t})(1-e^{-Dt})^{-1}t^{s-1} (\log t)dt\right|^2. 
$$
 We start by bounding $\mathcal{J}_2$ as it is simpler. Since $F_D(e^{-t})\ll e^{-t}$ for $t\geq 1$ we see that 
$$ \int_1^{\infty}F_D(e^{-t})(1-e^{-Dt})^{-1}t^{s-1} (\log t)dt \ll \int_1^{\infty} e^{-t} (\log t) dt =O(1). $$
Therefore, we deduce that 
\begin{equation}\label{Eq.FirstEstimateJ2}
    \mathcal{J}_2  \ll \sum_{D\in \G(x)} \left(\int_{Y^{-1}}^{1} |F_D(e^{-t})t^{-1/2} \log t|dt\right)^2 +x, 
\end{equation}
since $(1-e^{-Dt})^{-1} t^{s-1/2}= 1+o(1)$ for $Y^{-1}\leq t\leq 1$.  Now, by the Cauchy-Schwarz inequality we obtain 
\begin{align*}
\left(\int_{Y^{-1}}^{1} |F_D(e^{-t})t^{-1/2} \log t|dt\right)^2 &\leq \left(\int_{Y^{-1}}^1 \frac{(\log t)^2}{t} dt\right)\left(\int_{Y^{-1}}^{1}|F_D(e^{-t})|^2 dt\right) \\
& \ll (\log Y)^3 \int_{Y^{-1}}^{1}|F_D(e^{-t})|^2 dt.
\end{align*}
Furthermore, note that for $t\in (0,1)$ we have
\begin{equation}\label{Eq.TruncationFeketeT}
    F_D(e^{-t})= \sum_{n\leq D} \chi_D(n)e^{-nt}= \sum_{n\leq 2\log(1/t)/t} \chi_D(n)e^{-nt}+O(1).
\end{equation}
Inserting these estimates in \eqref{Eq.FirstEstimateJ2} gives 
\begin{equation}\label{Eq.SecondEstimateJ2}
    \mathcal{J}_2  \ll (\log Y)^3 \int_{Y^{-1}}^{1} \sum_{D\in \G(x)}\Big|\sum_{n\leq 2\log(1/t)/t} \chi_D(n)e^{-nt}\Big|^2dt +x(\log Y)^3. 
\end{equation}
We now use the large sieve inequality \eqref{Eq.BabyLargeSieve} which gives 
\begin{align}\label{Eq.BabySieveFekete}
  \sum_{D\in \G(x)} \Big|\sum_{n\leq 2\log(1/t)/t} \chi_D(n)e^{-nt}\Big|^2 &\ll \left(x + \frac{|\log t|^3}{t^2}\right) \sum_{\substack{m, n \leq 2\log(1/t)/t\\ mn= \square}} e^{-(n+m)t} \nonumber \\
  &\ll x\frac{(\log t)^4}{t},  
\end{align}
since $|\log t|^3/t^2\leq x$ in our range of $t$ and
$$\sum_{\substack{m, n \leq 2\log(1/t)/t\\ mn= \square}} e^{-(n+m)t}\leq \sum_{\substack{m, n \leq 2\log(1/t)/t \\ mn= \square}} 1\leq \sum_{j\leq 2\log(1/t)/t} \tau(j^2)\ll \frac{(\log t)^4}{t}.$$
 Combining the estimates \eqref{Eq.SecondEstimateJ2} and \eqref{Eq.BabySieveFekete} we derive 
\begin{equation}\label{Eq.ThirdEstimateJ2}
    \J_2 \ll x(\log Y)^3 \int_{Y^{-1}}^1 \frac{(\log t)^4}{t}dt + x(\log Y)^3 \ll x(\log Y)^8.
\end{equation}

 We now turn our attention to bounding $\mathcal{J}_1$.
 Using the P\'{o}lya-Vinogradov inequality and partial summation we obtain 
\begin{align*} F_D(e^{-t}) = \sum_{1 \leq n \leq D } \chi_D(n)e^{-nt} & =  \int_{1}^{D}\left(\sum_{1 \leq n \leq u} \chi_D(n) \right)  t e^{-tu} du \\
& \ll  \sqrt{D} (\log D) (1-e^{-Dt}). 
\end{align*} Hence we get
\begin{equation}\label{Eq.RangeTiny}
  \begin{aligned}
 \int_{0}^{\frac{(\log D)^2}{D}} F_D(e^{-t})(1-e^{- D t})^{-1}t^{s-1}(\log t) dt & \ll \sqrt{D}(\log D) \int_{0}^{(\log D)^2/D} t^{s-1}|\log t| dt \\
 & \ll \frac{(\log D)^{O(1)}}{D^{1/K}} = O(1),
\end{aligned}   
\end{equation}
 for all $D\in \G(x)$.
Next in the range $t\geq (\log D)^2/D$ we have  $(1-e^{-Dt})^{-1}\ll 1$ and hence
$$ \int_{(\log D)^2/D}^{Z^{-1}} F_D(e^{-t})(1-e^{-Dt})^{-1}t^{s-1} (\log t)dt \ll (\log x) Z^{-1/K} \int_{(\log D)^2/D}^{Z^{-1}} |F_D(e^{-t})| t^{-1/2}dt.$$
Therefore, by the Cauchy-Schwarz inequality we get 
\begin{equation}\label{Eq.L2Fekete}
\begin{aligned}
&\sum_{D\in \G(x)} \left|\int_{(\log D)^2/D}^{Z^{-1}} F_D(e^{-t})(1-e^{-Dt})^{-1}t^{s-1} (\log t)dt\right|^2 \\
& \ll(\log x)^2 Z^{-2/K} \sum_{D\in \G(x)} \left(\int_{(\log D)^2/D}^{Z^{-1}}  |F_D(e^{-t})|^2 dt\right)\left(\int_{(\log D)^2/D}^{Z^{-1}} \frac{1}{t}dt\right)  \\
&\ll (\log x)^3 Z^{-2/K}
 \sum_{D\in \G(x)} \int_{(\log D)^2/D}^{Z^{-1}}  |F_D(e^{-t})|^2 dt.
\end{aligned}
\end{equation}
We first handle the range $ x^{-1/2}(\log x)^{-4}\leq t\leq Z^{-1}$. In this range we use \eqref{Eq.TruncationFeketeT} and the large sieve inequality \eqref{Eq.BabyLargeSieve}. Similarly to \eqref{Eq.BabySieveFekete} this gives 
\begin{equation}\label{Eq.TruncationFeketeT2}
\begin{aligned}
    \sum_{D\in \G(x)}|F_D(e^{-t})|^2 &\ll \sum_{D\in \G(x)} \Big|\sum_{n\leq 2\log(1/t)/t} \chi_D(n)e^{-nt}\Big|^2 +x\\
    & \ll \left(x + \frac{|\log t|^3}{t^2}\right) \sum_{\substack{m, n \leq 2\log(1/t)/t\\ mn= \square}} e^{-(n+m)t}\ll \frac{x(\log x)^{15}}{t}.
\end{aligned}   
\end{equation}
Thus, we obtain 
\begin{equation}\label{Eq.FirstIntegralFekete}
    \sum_{D\in \G(x)} \int_{x^{-1/2}(\log x)^{-4}}^{Z^{-1}}  |F_D(e^{-t})|^2 dt \ll x(\log x)^{16}.
\end{equation}
Now in the remaining range $(\log D)^2/D\leq t\leq x^{-1/2}(\log x)^{-4}$ we shall use Lemma \ref{truncatedl}, but first we make a change of variables in the integral on the right hand side of \eqref{Eq.L2Fekete}. Indeed letting $u=1/(Dt)$ we obtain 
\begin{equation}\label{Eq.PoissonIntegralFekete}
\int_{(\log D)^2/D}^{ x^{-1/2}(\log x)^{-4}} |F_D(e^{-t})|^2 dt = \frac{1}{D}\int_{x^{1/2}(\log x)^4/D}^{1/(\log D)^2} \frac{|F_D(e^{-1/(Du)})|^2}{u^2} du.
\end{equation}
Now, by Lemma \ref{truncatedl} we obtain
$$ F_D(e^{-1/(Du)})= 2u\sqrt{D} \sum_{n\geq 1} \frac{\chi_D(n)}{1+4\pi^2(nu)^2}+ O(Due^{-1/u}).$$
Inserting this estimate in  \eqref{Eq.PoissonIntegralFekete} gives 
\begin{equation}\label{Eq.PoissonIntegralFekete2}
    \begin{aligned}
    \int_{(\log D)^2/D}^{ x^{-1/2}(\log x)^{-4}} |F_D(e^{-t})|^2 dt &\ll \int_{x^{-1/2}(\log x)^4}^{2/(\log x)^2} \Big|\sum_{n\geq 1} \frac{\chi_D(n)}{1+4\pi^2(nu)^2}\Big|^2 du+ D\int_{x^{-1/2}(\log x)^4}^{2/(\log x)^2} e^{-2/u} du\\
    & \ll \int_{x^{-1/2}(\log x)^4}^{2/(\log x)^2} \Big|\sum_{n\geq 1} \frac{\chi_D(n)}{1+4\pi^2(nu)^2}\Big|^2 du+1,
\end{aligned}
\end{equation}
since $e^{-2/u}\leq e^{-(\log x)^2}$ in our range of integration. Therefore, it only remains to bound the moment 
$$ \sum_{D\in \G(x)}\Big|\sum_{n\geq 1} \frac{\chi_D(n)}{1+4\pi^2(nu)^2}\Big|^2$$
for $u\in [x^{-1/2}(\log x)^4, 2/(\log x)^2]$.
We split the sum over $n$ into three parts: $n\leq 1/u$, $1/u<n\leq 1/u^2$ and $n>1/u^2$, and use the Cauchy-Schwarz inequality to reduce the problem to that of bounding the second moment of each of these three sums. Bounding $\chi_D(n)$ trivially, we find that the contribution of the last part is 
\begin{equation}\label{Eq.ThirdPartPoissonFek}
    \sum_{D\in \G(x)} \Big|\sum_{n> 1/u^2} \frac{\chi_D(n)}{1+4\pi^2(nu)^2}\Big|^2 \ll x\Big(\sum_{n> 1/u^2} \frac{1}{(nu)^2}\Big)^2 \ll x.
\end{equation}
We now use the large sieve inequality \eqref{Eq.BabyLargeSieve} to bound the contribution of the first part. This gives 
\begin{equation}\label{Eq.FirstPartPoissonFek}\sum_{D\in \G(x)} \Big|\sum_{n\leq  1/u} \frac{\chi_D(n)}{1+4\pi^2(nu)^2}\Big|^2 \ll \left(x+ \frac{(\log(1/u)}{u^2}\right)\sum_{\substack{m, n \leq 1/u\\ mn=\square}} 1 \ll \frac{x \log(1/u)^3}{u}.
\end{equation}
Finally, to bound the contribution of the second part we use the large sieve inequality \eqref{Eq.BabyLargeSieve2} which gives in this case
\begin{equation}\label{Eq.SecondPartPoissonFek}
\begin{aligned}\sum_{D\in \G(x)} \Big|\sum_{1/u<n\leq  1/u^2} \frac{\chi_D(n)}{1+4\pi^2(nu)^2}\Big|^2 &\ll 
x\sum_{\substack{1/u< m, n \leq 1/u^2\\ mn=\square}} \frac{1}{u^4(mn)^2}+ \log (1/u) \left(\sum_{1/u<n\leq 1/u^2} \frac{n^{1/2}}{n^2u^2}\right)^2\\
& \ll \frac{x}{u^4}\sum_{j> 1/u}\frac{\tau(j^2)}{j^4} + \frac{\log(1/u)}{u^3}\\
& \ll \frac{x \log(1/u)^3}{u},
\end{aligned}
\end{equation}
in our range of $u$.
Combining the bounds \eqref{Eq.ThirdPartPoissonFek}, \eqref{Eq.FirstPartPoissonFek} and \eqref{Eq.SecondPartPoissonFek} and using the Cauchy-Schwarz inequality implies that 
$$\sum_{D\in \G(x)} \Big|\sum_{n\geq 1} \frac{\chi_D(n)}{1+4\pi^2(nu)^2}\Big|^2 \ll  \frac{x \log(1/u)^3}{u}.$$
Inserting this estimate in \eqref{Eq.PoissonIntegralFekete2} we obtain 
$$\sum_{D\in \G(x)}\int_{(\log D)^2/D}^{ x^{-1/2}(\log x)^{-4}} |F_D(e^{-t})|^2 dt \ll x(\log x)^4.$$
Finally, we combine this estimate with \eqref{Eq.RangeTiny}, \eqref{Eq.L2Fekete} and \eqref{Eq.FirstIntegralFekete} to obtain 
$$ \mathcal{J}_1\ll x.$$
Using this bound together with \eqref{Eq.ThirdEstimateJ2} and proceeding as in \eqref{Eq.DeduceTruncationChar} completes the proof. 

 \end{proof}
 
\begin{proof}
Since the proof of Theorem \ref{MainFek} is exactly the same as that of Theorem \ref{ThmPartialSumsPositive} we only indicate where the main changes occur. Indeed, we choose the same parameters $R, M, Y_1, Y_2$, as well as the same points $(s_r)_{R/5\leq r\leq R}$ as in the proof of Theorem \ref{ThmPartialSumsPositive}. We then replace condition $3$ for our set of discriminants $\mathcal{H}(x)$ by the following condition: 
$$ \left|\int_{0}^{\infty}\frac{F_D(e^{-t})}{1-e^{-Dt}}t^{s_r-1}(\log t) dt-\int_{e^{-Y_2}}^{e^{-Y_1}} \frac{F_D(e^{-t})}{1-e^{-Dt}}t^{s_r-1}(\log t) dt\right|\leq Y_1^5,$$
for all $R/5\leq r\leq R$. We also replace the identity \eqref{Eq.LaplaceCharacters} by \eqref{Eq.IdentityLaplaceFekete} and note that
$\frac{\Gamma'}{\Gamma}(s) \ll 1$ for $1/2 \leq s \leq 1$. Then, it follows from Propositions \ref{CLTLogL} and \ref{pro.TruncationLaplaceFek} and Theorem \ref{Thm.SignChangesL'L}, that for all positive discriminants $D\in \F$, except for a set of size $O\left(x\exp\left(-(\log_3 x)^A\right)\right)$
we have 
$$S^{-}\left(\widetilde{\mathcal{L}}_D(s), 0, \infty\right)>\frac{\delta}{5} R,$$
where 
$$\widetilde{\mathcal{L}_D}(s)= \int_{e^{-Y_2}}^{e^{-Y_1}} \frac{F_D(e^{-t})}{1-e^{-Dt}}t^{s-1}(\log t) dt. $$
Using the change of variables $u=-\log t$ one can see that $\widetilde{\mathcal{L}_D}$ 
is the Laplace transform of the function
$$
\tilde{g}(u)=\begin{cases} -uF_D(e^{-e^{-u}}) \big(1-e^{-De^{-u}}\big)^{-1} &\text{ if } Y_1\leq u\leq Y_2,\\
 0 &\text{ otherwise. }\end{cases}
 $$
Appealing to Lemma \ref{lem:Karlin} and noting that $e^{-e^{-u}}=1-e^{-u}+O(e^{-2u})$ for large $u$ completes the proof. 
    
\end{proof}

\subsection{Real zeros of the Theta function associated to $\chi_D$}
A standard computation using Mellin transforms implies that for $\Re(s) >0,$
\begin{equation*}\Gamma(s/2)L(s,\chi_D)= \int_{0}^{\infty} \theta\left(\frac{tD}{\pi},\chi_D\right)t^{s/2} \frac{dt}{t}.\end{equation*}
 Upon taking logarithmic derivatives we deduce
\begin{equation}\label{Mellinder} \Gamma(s/2)L(s,\chi_D) \left(\frac{L'}{L}(s,\chi_D))+\frac{1}{2}\frac{\Gamma'}{\Gamma}(s/2)\right)=\int_{0}^{\infty} \theta\left(\frac{tD}{\pi},\chi_D\right)t^{s/2-1} (\log t)dt. \end{equation}
Clearly $\frac{\Gamma'}{\Gamma}(s/2) \ll 1$ for $1/4 \leq s \leq 1$ and the proof of Theorem \ref{Main} implies that the left hand side of \eqref{Mellinder} changes sign at least $\gg \log_2 x/\log_4 x$ times, for  all discriminants $D \in \F$  except for a set of size $\ll  x\exp\left(-(\log_3 x)^A\right).$ Noting that the factor $t^{s/2-1} (\log t)$ changes sign only once in $(0,\infty)$, the result follows from Lemma \ref{lem:Karlin}.

\section{Upper bounds on the number of zeros of Fekete polynomials}\label{upperbounds}

\subsection{Real zeros of polynomials}
One of our main tools is the following consequence of Jensen's formula. The number of zeros of a polynomial $P(z)$ inside the circle $\vert z-z_0\vert =r$ centered at $z_0$ does not exceed 

\begin{equation}\label{Jensen} \left(\log \frac{\max_{\vert z-z_0\vert =R} \vert P(z)\vert}{\vert P(z_0)\vert} \right)/ \log (R/r). \end{equation} We  also record a general result obtained by Borwein, Erd\'elyi and K\'{o}s.

\begin{lem}\cite[Theorem $4.2$]{ErdBor}\label{leftzeros} There exists an absolute constant $c$ such that every polynomial $p$ of the form
$$p(z)=\sum_{k=0}^{n}a_k z^k, \hspace{2mm} \vert a_j\vert \leq 1, \vert a_0\vert = 1, a_j \in \mathbb{C}$$ has at most $\frac{c}{a}$ zeros in $(-1+a,1-a)$ whenever $a\in (0,1)$. \end{lem}

Our goal in this section is to prove Theorem \ref{Thmupperbnd}.
In order to do so, we need to show that the Fekete polynomial is often not too ``small' at several well-chosen points before applying Jensen's formula \eqref{Jensen}. As with many non-vanishing problems, our approach relies on the computation of the first two moments.

Our starting point is an observation that the Fekete polynomial evaluated at $z_{\alpha}:=\exp(-1/x^{\alpha})$ resembles a character sum of length $\approx x^{\alpha}$. 

\subsection{Covering the real line with three circles}

To this end, we shall consider, for $\varepsilon>0$, the Fekete polynomial $F_D$ evaluated at the points
$\exp(-x^{-1/4+\varepsilon}),\exp(-x^{-1/2})$ and $\exp(-x^{1/4}/D)$ (this choice will become clear from the proof of Theorem \ref{Thmupperbnd}).

Using Lemma \ref{truncatedl}, we can pass from a sum of approximate length $D/x^{1/4}$ to a dual sum of approximate length $x^{1/4}$. Indeed, for $0<D\leq x$ we obtain

\begin{equation}\label{poissondev} F_D\left(\exp\left(-\frac{x^{1/4}}{D}\right)\right)= \frac{2\sqrt{D}}{x^{1/4}} \sum_{k\geq 1} \frac{\chi_D(k)}{1+4\pi^2(kx^{-1/4})^2} + O(1).\end{equation} 
Let $\F^+$ be the set of positive fundamental discriminants $D
\leq x$ and define for $j=1,2$ the following mixed moments

$$\mathcal{S}_j:=\sum_{D\in \F^+} \left(F_D(\exp(-x^{-1/4+\varepsilon}))F_D(\exp(-x^{-1/2}))F_D(\exp(-x^{1/4}/D))\right)^j.$$ Our main result will follow from a lower bound on the first moment and an upper bound on the second moment. 

\begin{pro}\label{thirdmoment}For large $x$ we have  $\mathcal{S}_1 \gg x^{7/4-\varepsilon/2}$.    \end{pro}
\begin{proof} 
We first truncate the summation and write
\begin{equation}\label{eq:truncation1} F_D(\exp(-1/x^{1/2})) = \sum_{m\leq x^{1/2+o(1)}} \chi_D(m)\exp(-m/x^{1/2}) + O(1/x^{50}),\end{equation} and similarly 
\begin{equation}\label{eq:truncation2} F_D(\exp(-1/x^{1/4-\varepsilon})) = \sum_{n\leq x^{1/4-\varepsilon+o(1)}} \chi_D(n)\exp(-n/x^{1/4-\varepsilon}) + O(1/x^{50}).\end{equation} Using these estimates together with  \eqref{poissondev} we obtain
 \begin{align*} \mathcal{S}_1 = \frac{2}{x^{1/4}}\sum_{\ell }c(\ell) \sum_{D\in \F^+}\chi_D(\ell)\sqrt{D}  + O(x^{7/4-\varepsilon} (\log x)^2)\end{align*}  where the coefficients $c(\ell)$ are defined by
 \begin{equation}\label{def-coeffc}
     c(\ell):= \sum_{\substack{mnk=\ell\\ m \leq x^{1/2+o(1)}, \ n\leq x^{1/4-\varepsilon+o(1)}}} \frac{e^{-nx^{\varepsilon}/x^{1/4}}e^{-m/x^{1/2}}} {1+4\pi^2(kx^{-1/4})^2}. \end{equation}
 In particular we have
\begin{equation}\label{behavior-coeffc}
  \sum_{mnk= \ell \atop m \leq x^{1/2+o(1)}, \ n\leq x^{1/4-\varepsilon+o(1)}, k\leq x^{1/4}} 1 \ll   c(\ell) \ll \sum_{mnk= \ell \atop m \leq x^{1/2+o(1)}, \ n\leq x^{1/4-\varepsilon+o(1)}} \min(1,x^{1/2}/k^2).
\end{equation}


  We split the summation into two parts 
 $$\mathcal{S}_1^{sq}:=\sum_{ \ell=\square }c(\ell) \sum_{D\in \F^+}\chi_D(\ell)\sqrt{D}  \hspace{5mm}\text{ and }\hspace{5mm}\mathcal{S}_1^{nsq}:=\sum_{ \ell \neq \square }c(\ell) \sum_{D\in \F^+}\chi_D(\ell)\sqrt{D}$$ depending on whether $\ell$ is a square or not. First, observe that if $\ell=mnk \leq x$ is a square, we have by Lemma \ref{Orthogonality} and partial summation
 \begin{align}\label{eq:partortho}
\sum_{D\in \F^+} \chi_D(nmk) \sqrt{D} & \gg x^{3/2}\prod_{p\mid mnk}\frac{p}{p+1} + O(x (mnk)^{o(1)}) \nonumber \\
& \gg x^{3/2}\left\{\frac{\varphi(m)}{m}\frac{\varphi(n)}{n}\frac{\varphi(k)}{k} \right\}. \end{align}
It follows using \eqref{behavior-coeffc} that \begin{align*}\mathcal{S}_1^{sq} & \gg x^{3/2} \sum_{m \leq x^{1/2}, n\leq x^{1/4-\varepsilon}, k\leq x^{1/4}\atop m,n,k=\square} \left(\frac{\varphi(m)}{m}\frac{\varphi(n)}{n}\frac{\varphi(k)}{k} \right) \\
& \gg x^{3/2} \sum_{ m\leq x^{1/4}, n\leq x^{1/8-\varepsilon/2}, k\leq x^{1/8}} \frac{\varphi(n^2)}{n^2}\frac{\varphi(m^2)}{m^2} \frac{\varphi(k^2)}{k^2} \\
& \gg  x^{3/2}x^{1/4}x^{1/8-\varepsilon/2}x^{1/8} = x^{2-\varepsilon/2}.
\end{align*} Mutiplying by $2x^{-1/4}$, this gives the expected main contribution to $\mathcal{S}_1$.


We now want to show that the non-square part gives a negligible contribution, that is $\mathcal{S}_1^{nsq}=o(x^{2-\varepsilon/2})$. 
By partial summation we have  \begin{equation}\label{eq:partsumm}  \sum_{D \in \F^+} \chi_D(\ell) \sqrt{D} \ll x^{1/2} \sum_{D \in \F^+} \chi_D(\ell) + \int_{1}^{x}  \left(\sum_{D\in \mathcal{F}(t)^+} \chi_D(\ell)\right) t^{-1/2}dt.\end{equation}

We split the sum over $\ell$ as $\mathcal{S}_1^{nsq} = \mathcal{S}_{1}^{<}+\mathcal{S}_1^{>}$  depending on whether $\ell \leq x^{1-\varepsilon}$ or not. In the former case, we use the trivial bound $c(\ell) \leq \sum_{\ell=mnk} 1 \ll  \ell^{o(1)}$ and the Cauchy-Schwarz inequality to get
\begin{align}\label{firstpart-sum} x^{1/2} \sum_{\ell \neq \square \atop \ell \leq x^{1-\varepsilon}} c(\ell) \left\vert \sum_{D \in \F^+} \chi_D(\ell) \right\vert & \ll  x^{1/2+o(1)} \left( \sum_{\ell \leq x^{1-\varepsilon}} 1 \right)^{1/2}\left( \sum_{\ell \neq \square \atop \ell \leq x^{1-\varepsilon}} \left\vert \sum_{D \in \F^+} \chi_D(\ell) \right\vert^2\right)^{1/2} \nonumber \\
& \ll x^{2-\varepsilon+o(1)},\end{align} where we applied a suitable version of Lemma \ref{lemmaJutila} where $\F$ is replaced by $\F^+$  in the last step. Similarly, applying the Cauchy-Schwarz inequality twice and swapping the summation over $\ell$ and the integration we get 
\begin{align}\label{Eq.SumCSTwice}
 \sum_{\ell \neq \square \atop \ell \leq x^{1-\varepsilon}} c(\ell) \left\vert \int_{1}^{x}  \left( \sum_{D\in \mathcal{F}(t)^+} \chi_D(\ell)\right) \frac{1}{\sqrt{t}}dt \right\vert & \ll x^{\frac{1-\varepsilon}{2}+o(1)}\left( \sum_{\ell \neq \square \atop \ell \leq x^{1-\varepsilon}} \left\vert \int_{1}^{x}  \left( \sum_{D\in \mathcal{F}(t)^+} \chi_D(\ell)\right) \frac{1}{\sqrt{t}}dt \right\vert^2 \right)^{1/2} \nonumber\\
 & \ll x^{\frac{1-\varepsilon}{2}+o(1)}\left(   \int_{1}^{x} \sum_{\ell \neq \square \atop \ell \leq x^{1-\varepsilon}} \left| \sum_{D\in \mathcal{F}(t)^+} \chi_D(\ell)\right|^2 dt  \right)^{1/2}. 
\end{align} 
Appealing to Lemma \ref{lemmaJutila} implies that the right hand side of the above estimate is 
\begin{equation}\label{secondpart-sum} 
\ll x^{2-\varepsilon+o(1)}.  \end{equation}
Combining \eqref{eq:partsumm}, \eqref{firstpart-sum} and \eqref{secondpart-sum} we obtain that $\mathcal{S}_1^{<} \ll x^{2-\varepsilon+o(1)},$ which is an acceptable contribution.

We now turn our attention to the case $\ell \geq x^{1-\varepsilon}.$ Remark that under the conditions $m \leq x^{1/2+o(1)}, n\leq x^{1/4-\varepsilon+o(1)}$, we have 
$$
k=\frac{\ell}{mn} \geq \frac{\ell}{x^{3/4-\varepsilon+o(1)}},
$$
and thus in this case
$$ c(\ell) \ll \sum_{mnk= \ell \atop m \leq x^{1/2+o(1)}, \ n\leq x^{1/4-\varepsilon+o(1)}} \frac{x^{1/2}}{k^2} \ll x^{2-2\varepsilon+o(1)} \ell^{-2+o(1)}$$
by \eqref{behavior-coeffc}.  Therefore we deduce that  
the contribution to $\mathcal{S}_1^{nsq}$  of the positive integers $\ell\geq x^{1-\varepsilon}$ is 
\begin{equation}\label{contribSnq} \ll x^{2-2\varepsilon+o(1)} \sum_{\substack{\ell\geq x^{1-\varepsilon} \\ \ell\neq \square}} \frac{\ell^{o(1)}}{\ell^2} \left\vert \sum_{D \in \F^+} \chi_D(\ell) \sqrt{D}\right\vert.   \end{equation} 
By \eqref{eq:partsumm}, the sum on the right hand side of \eqref{contribSnq} is bounded by $x^{2-2\varepsilon+o(1)}(\mathcal{S}_{1,1} + \mathcal{S}_{1,2})$ where (after a dyadic summation)
\begin{equation}\label{defS11} \mathcal{S}_{1,1}=  x^{1/2}\sum_{j \geq \frac{(1-\varepsilon) \log x}{\log 2}}\frac{(2^{j})^{o(1)}}{2^{2j}}\sum_{\substack{2^j\leq \ell  < 2^{j+1} \\ \ell \neq \square}} \left\vert \sum_{D \in \F^+} \chi_D(\ell) \right\vert  \end{equation} and
\begin{equation}\label{defS12} \mathcal{S}_{1,2}=  \sum_{j \geq \frac{(1-\varepsilon) \log x}{\log 2}}\frac{(2^{j})^{o(1)}}{2^{2j}}\sum_{2^j \leq \ell <2^{j+1} \atop \ell \neq \square}  \left\vert \int_{1}^{x}  \left( \sum_{D\in\mathcal{F}(t)^+} \chi_D(\ell)\right) \frac{1}{\sqrt{t}}dt \right\vert. \end{equation}
Applying the Cauchy-Schwarz inequality and combining it with Lemma \ref{lemmaJutila} we get
\begin{align*}x^{1/2}\sum_{\substack{2^j\leq \ell  < 2^{j+1} \\ \ell \neq \square}} \left\vert \sum_{D \in \F^+} \chi_D(\ell) \right\vert & \ll  x^{1/2}\left(\sum_{2^j\leq \ell  < 2^{j+1}} 1\right)^{1/2} \left(\sum_{\substack{2^j\leq \ell  < 2^{j+1} \\ \ell \neq \square}} \left\vert \sum_{D \in \F^+} \chi_D(\ell)\right\vert^2\right)^{1/2} \\
& \ll  x^{1+o(1)} 2^{j(1+o(1))}.\end{align*}
Furthermore, using the same  argument leading to \eqref{Eq.SumCSTwice} we obtain \begin{align*}
 \sum_{2^j \leq \ell <2^{j+1} \atop \ell \neq \square}  \left\vert \int_{1}^{x}  \left( \sum_{D\in\mathcal{F}(t)^+} \chi_D(\ell)\right) \frac{1}{\sqrt{t}}dt \right\vert & \ll x^{1+o(1)}2^{j(1+o(1))}. 
\end{align*}

Summing over the dyadic intervals, we finally get
\begin{equation}\label{boundS112}
      \mathcal{S}_{1,1}+ \mathcal{S}_{1,2} \ll  \sum_{j \geq \frac{(1-\varepsilon) \log x}{\log 2}} \frac{x^{1+o(1)}}{2^{j(1+o(1))}} \ll x^{\varepsilon+o(1)}. \end{equation} Thus, by \eqref{contribSnq} and \eqref{boundS112}
      we have $\mathcal{S}_1^{>} \ll x^{2-2\varepsilon+o(1)} (\mathcal{S}_{1,1} + \mathcal{S}_{1,2}) \ll x^{2-\varepsilon+o(1)}$ which is acceptable.  This concludes the proof.

\end{proof} 
 The following proposition gives an almost optimal upper bound for the second moment.
\begin{pro}\label{thirdmoment-2}
For large $x$ we have $\mathcal{S}_{2} \ll x^{5/2+o(1)}. $
\end{pro}

\begin{proof}
First, we split the sum from \eqref{poissondev} using dyadic summation and write 
\begin{align*} &F_D(\exp(-x^{1/4}/D))\\
&= \frac{2 \sqrt{D}}{x^{1/4}} \left( \sum_{ 1\leq k\leq x^{1/4}} \chi_D(k) h(k) + \sum_{i \atop  x^{1/4} < 2^i \leq x^{10}}\sum_{ 2^i < k \leq 2^{i+1}} \chi_D(k) h(k) + O\left(x^{1/2}\sum_{k \geq x^{10}} \frac{1}{k^2}\right) \right)  \\
& = \frac{2 \sqrt{D}}{x^{1/4}} \left( \sum_{ 1\leq k\leq x^{1/4}} \chi_D(k) h(k) + \sum_{i \atop  x^{1/4} <2^i \leq x^{10}}\sum_{ 2^i < k \leq 2^{i+1}} \chi_D(k) h(k) \right) + O(1/x^9),
\end{align*} where $h$ is a function such that $h(k) \ll \min(1,x^{1/2}/k^2).$ By the Cauchy-Schwarz inequality (applied twice) we have 
\begin{align}\label{eq:dyadic} &\left( \sum_{ 1\leq k\leq x^{1/4}} \chi_D(k) h(k) + \sum_{i \atop  x^{1/4} < 2^i \leq x^{10}}\sum_{ 2^i < k \leq 2^{i+1}} \chi_D(k) h(k)\right)^2 \nonumber \\ & \ll  \left( \sum_{ 1\leq k\leq x^{1/4}} \chi_D(k) h(k)\right)^2   + x^{o(1)} \sum_{i \atop  x^{1/4} < 2^i \leq x^{10}} \left(\sum_{ 2^i < k \leq 2^{i+1}} \chi_D(k) h(k)\right)^2.\end{align}
Hence, it is sufficient to bound separately the contribution to $\mathcal{S}_2$ of each of these sums. 
 Expanding the square, using \eqref{eq:truncation1}, \eqref{eq:truncation2} and the Cauchy-Schwarz inequality,  we obtain  \begin{align*}  S_2^1:&=\sum_{D \in \F^+}  \left(F_D(\exp(-1/x^{1/2}))  F_D(\exp(-1/x^{1/4-\varepsilon}))\right)^2\left(\frac{2 \sqrt{D}}{x^{1/4}}\sum_{ 1\leq k\leq x^{1/4}} \chi_D(k) h(k)\right)^2 \\
 &\ll x^{1/2}\sum_{D \in \F^+} \left(\sum_{\ell }\chi_D(\ell) a(\ell) \right)^2 \ +  1
 \end{align*} where the coefficients $a(\ell)$ satisfy the bound
\begin{equation}\label{bound_a}
     a(\ell) \ll \sum_{mnk=\ell \atop m\leq x^{1/2+o(1)}, n\leq x^{1/4+o(1)},k \leq x^{1/4}} 1 .\end{equation}
Note that the divisor bound implies $a(\ell) \ll \ell^{o(1)}.$ Hence, by Lemma \ref{HB-largesieve} we obtain 
\begin{align}\label{boundS21}
    S_2^1  &\ll x^{3/2+o(1)} \sum_{\ell_1,\ell_2 \leq x^{1+o(1)} \atop \ell_1 \ell_2 = \square} a(\ell_1)a(\ell_2) \ +1 \ll x^{3/2+o(1)} \sum_{j \leq x^{1+o(1)}} \tau(j^2) \nonumber \\
    & \ll x^{5/2+o(1)}.\end{align}
    Similarly, we have 
 \begin{align*} S_2^i & := \sum_{D \in \F^+}  \left(F_D(\exp(-1/x^{1/2}))  F_D(\exp(-1/x^{1/4-\varepsilon})\right)^2\frac{D}{x^{1/2}}\left(\sum_{ 2^i \leq k <2^{i+1}} \chi_D(k) h(k)\right)^2 \\
&   \ll x^{1/2}\sum_{D \in \F^+} \left(\sum_{\ell }\chi_D(\ell) a_i(\ell) \right)^2 +1
 \end{align*} where the coefficients $a_i$ satisfy the bound
 $$ a_i(\ell) \ll x^{1/2} \sum_{mnk=\ell \atop m\leq x^{1/2+o(1)}, n\leq x^{1/4+o(1)},2^i <k \leq 2^{i+1}} 1/k^2 . $$ Again, by the divisor bound we obtain $a_i(\ell) \ll \frac{1}{2^{2i}} \ell^{o(1)}.$ Another application  of Lemma \ref{HB-largesieve} gives 

\begin{align}\label{boundS2i}
    S_2^i  &\ll x^{1/2+o(1)}(x+2^i x^{3/4+o(1)}) \sum_{\ell_1,\ell_2 \leq 2^i x^{3/4+o(1)} \atop \ell_1 \ell_2 = \square} a_i(\ell_1)a_i(\ell_2) \ +1\nonumber \\
    & \ll x^{3/2+o(1)} (x+2^i x^{3/4+o(1)}) 2^{-4i}\sum_{j \leq 2^ix^{3/4+o(1)}} \tau(j^2) \nonumber\\
    & \ll x^{3/2+o(1)} (x+2^i x^{3/4+o(1)}) 2^{-4i} (2^i x^{3/4+o(1)}) =  x^{13/4+o(1)}2^{-3i} + x^{3+o(1)}2^{-2i}.\end{align}
Using \eqref{eq:dyadic}, we now sum over the dyadic intervals incorporating the bounds \eqref{boundS21} and \eqref{boundS2i} to arrive at
\begin{align*}\mathcal{S}_2 &\ll S_2^1 +  \sum_{\substack{i\\ x^{1/4} <2^i \leq x^{10}}} S_2^i \ +1 \\
&\ll x^{5/2+o(1)} + x^{13/4+o(1)}  \sum_{\substack{i\\ x^{1/4} <2^i \leq x^{10}}} 2^{-3i} + x^{3+o(1)} \sum_{\substack{i\\ x^{1/4} <2^i \leq x^{10}}} 2^{-2i}  \\
&\ll x^{5/2+o(1)}.
\end{align*}
This concludes the proof.
    
\end{proof} We immediately deduce the following.
\begin{Cor}\label{cor3circ}
There exists at least $\gg x^{1-o(1)}$ fundamental discriminants $D\in\F^+$ such that
$$\min\{\left|F_D(\exp(-x^{-1/4+\varepsilon}))\right|,\left|F_D(\exp(-x^{-1/2}))\right|,\left|F_D(\exp(-x^{1/4}/D))\right|\} \gg x^{-100}.$$
  \end{Cor}
\begin{proof}
Let us fix some $\varepsilon>0$. We define $\mathcal{L}(x)$ to be the set of fundamental discriminants $D\in\F^+$ such that $$\min\{\left|F_D(\exp(-x^{-1/4+\varepsilon}))\right|,\left|F_D(\exp(-x^{-1/2}))\right|,\left|F_D(\exp(-x^{1/4}/D))\right|\} \gg x^{-100}$$ and $\mathcal{S}_1^{*}$ be the sum $\mathcal{S}_1$  restricted to discriminants $D$ in $\mathcal{L}(x).$
Note that $\mathcal{S}_1^{*}= \mathcal{S}_1 (1+o(1))$ by Proposition \ref{thirdmoment}.
Applying the Cauchy-Schwarz inequality we get $$\mathcal{S}_1^{*} \ll \# \mathcal{L}(x)^{1/2} \mathcal{S}_2^{1/2}.$$ The conclusion now follows from Propositions \ref{thirdmoment} and \ref{thirdmoment-2}.
\end{proof}
We conclude this section by proving Theorem \ref{Thmupperbnd}.

\begin{proof}[Proof of Theorem \ref{Thmupperbnd}]
Let $\varepsilon>0$ and define $z_{1}=\exp(-x^{-1/4+\varepsilon})$, $z_2= \exp(-x^{-1/2})$ and $z_3= \exp(-x^{1/4}/D).$
 By Corollary \ref{cor3circ}, there exists at least $\gg x^{1-o(1)}$ fundamental discriminants $D\in\F^+$ such that 
\begin{equation}\label{mini} \min\{|F_D(z_1)|,|F_D(z_2)|,|F_D(z_3)|\} \gg x^{-100}.\end{equation}  Let $D$ be one such discriminant. We wish to show that $F_D$ has at most $O(x^{1/4+o(1)})$ real zeros.\\ 
We consider two concentric circles $C_{r_{1}},C_{R_{1}}$ centered at $z_{1}$ and of radii
 $$r_{1}=z_2-z_1 \,\,\,\textrm{ and } \,\,\, R_1=1-z_1.$$ 
Clearly we have $$\max_{z\in C_{R_{1}}} \vert F_D(z)\vert \leq \sum_{n\leq \vert D\vert} 1 \leq x$$ and 
$\log (R_{1}/r_{1}) \gg x^{-1/4-\varepsilon}$. By Jensen's formula \eqref{Jensen} and \eqref{mini}, the number of zeros inside $C_{r_{1}}$ is $O(x^{1/4+\varepsilon+o(1)})$. We further define two concentric circles $C_{r_{2}},C_{R_{2}}$ centered at $z_{2}$ and of radii $$r_{2}=z_3-z_2 \,\,\,\textrm{ and } \,\,\, R_2=1-z_2.$$ 
Similarly, we have $\max_{z\in C_{R_{2}}} \vert F_D(z)\vert \leq x$ and  $\log(R_{2}/r_{2}) \gg x^{-1/4}$. By \eqref{Jensen} and \eqref{mini}, we deduce that the number of real zeros inside $C_{r_{2}}$ of  $F_D$ is $O(x^{1/4+o(1)})$. Finally, we define the two circles $C_{r_{3}},C_{R_{3}}$ centered at  $z_{3}$ and of radii $$r_{3}=1-z_3 \,\,\,\textrm{ and } \,\,\, R_3=2r_3.$$ 
We have the bound $$\max_{z\in C_{R_{3}}} \vert F_D(z)\vert \ll \sum_{n=1}^{\vert D\vert} \left(1+\frac{1}{x^{3/4}}\right)^n \ll \exp(cx^{1/4}),$$ for some constant $c>0$. Similarly as above, this implies that the number of real zeros inside $C_{r_{3}}$ of $F_D$ is $O(x^{1/4+o(1)})$. Note that we have $$\left[z_{1},1\right] \subset C_{r_{1}}\cup C_{r_{2}}\cup C_{r_{3}}.$$ Moreover, Lemma~\ref{leftzeros} asserts that the number of zeros of $F_D$ in $(0,z_{1})$ is bounded by $O(1/(1-z_{1}))= O(x^{1/4})$. Combining all of the above and noting that the parameter $\varepsilon>0$ can be taken arbitrary small, the conclusion follows.\end{proof}

\section{Constructing Fekete polynomials with no zeros in $\left(0, 1-(\log x)^{-\sqrt{e}+\varepsilon}\right)$}
 In this section, we investigate the non-vanishing of Fekete polynomials in some subintervals of $(0,1)$ and prove Theorem \ref{NoZerosThm}. To this end, we shall only consider polynomials associated to fundamental discriminants $D$ with the nice property that $\chi_D(n)=1$ for all the ``small'' positive integers $n$. Indeed, we have the following lemma.

\begin{lem}\label{NoZerosLem}
Let $\varepsilon>0$ be a fixed small number. Let $x$ be large and $D\in \F$ be such that $\chi_D(n)=1$ for all $n\leq y$, where $y\to \infty$ as $x\to \infty$.  Then $F_D(z)$ does not have zeros in the interval $\left(0, 1-y^{-\sqrt{e}+\varepsilon}\right)$.
\end{lem}
\begin{proof}
We first use Vinogradov's trick to show that 
\begin{equation}\label{VTrick}
\sum_{n\leq t} \chi_D(n)\gg \varepsilon t
\end{equation}
for all $t\leq y^{\sqrt{e}-\varepsilon/2}.$
The proof is standard but we include it for the sake of completeness. Note that we might assume that $y\leq t \leq y^{\sqrt{e}-\varepsilon/2}$, since the estimate is trivial otherwise.
Let $\Psi(t, y)$ denote the number of $y$-smooth (or $y$-friable) integers up to $t$. Since $y\leq t<y^2$ we obtain
$$ \sum_{n\leq t} \chi_D(n) \geq \Psi(t, y) -\sum_{\substack{ n\leq t \\ \exists p \mid n, \  p>y}} 1= \lfloor t\rfloor - 2\sum_{y<p \leq t} \left \lfloor \frac{t}{p} \right \rfloor.$$
Using Mertens' theorem the right hand side equals
$$ t \left(1-2\sum_{y<p \leq t} \frac{1}{p}\right) +O\left(\frac{t}{\log t}\right)=
t \left(1-2\log \left(\frac{\log t}{\log y} \right)\right) +O\left(\frac{t}{\log y}\right)\gg \varepsilon t,$$
which establishes \eqref{VTrick}. 

Next, note that for any positive integer $k\leq |D|-1$ and any real number $z\in (0, 1)$ we have
\begin{equation}\label{TruncFek}
F_D(z)\geq \sum_{n=1}^k \chi_D(n) z^n - \sum_{n=k+1}^{\vert D\vert-1} z^n\geq \sum_{n=1}^k \chi_D(n) z^n  - \frac{z^{k+1}}{1-z}.
\end{equation}
Let us first suppose that $0<z < 1-1/y$. We choose $k=\lfloor y\rfloor$ in this case to get 
$$ F_D(z)\geq  \sum_{n=1}^k z^n- \frac{z^{k+1}}{1-z}
= z\frac{1-2 z^k}{1-z} >0,$$
since
$ z^{k}\leq (1-1/y)^y <1/2, $
if $y$ is large enough.

We now consider the case where $1-y^{-1}\leq z\leq 1-y^{-\sqrt{e}+\varepsilon}$, and choose $k=\lfloor -A/ \log z \rfloor+1$ for some suitably large constant $A>0$. Hence we have $k \ll y^{\sqrt{e}-\varepsilon} \leq y^{\sqrt{e}-\varepsilon/2}$, if $y$ is large enough. Therefore, by partial summation and \eqref{VTrick} we obtain
\begin{equation}\label{PartialFek}
\begin{aligned}
 \sum_{n=1}^k \chi_D(n) z^n
 & = z^{k} \sum_{n\leq k} \chi_D(n) - \log z \int_{1}^k z^t \left(\sum_{n\leq t} \chi(n)\right) dt\\
 & \geq c_0\left(\varepsilon k z^k - \varepsilon   \int_{1}^k  t \big(z^t \log z\big)  dt\right)  \\
 & = c_0\varepsilon \int_{1}^k  z^t  dt = c_0\varepsilon\frac{z^{k}- z}{\log z},
 \end{aligned}
 \end{equation}
 for some constant $0<c_0<1$.
 Writing $h=1-z$ and using that $-\log (1-h) <2h$ if $0<h<1/2$  we deduce from \eqref{TruncFek} and \eqref{PartialFek} that 
 $$ F_D(z) \geq \frac{c_0\varepsilon}{2} \frac{z-z^k}{1-z} - \frac{z^{k+1}}{1-z} \geq z\frac{c_0\varepsilon/2-2e^{-A}}{1-z} >0,
 $$
 if $A$ is suitably large. This completes the proof.

\end{proof}
To complete the proof of Theorem \ref{NoZerosThm} we construct ``many'' fundamental discriminants $0<D\leq x$ such that $\chi_D(n)=1$ for all $n\leq \log x$. 
\begin{lem}\label{DiscriminantsChiD1} Let $x$ be large, and $2\leq y\leq (\log x)^2$ be a real number. The number of fundamental discriminants $0<D\leq x$ such that $\chi_D(n)=1$ for all $n\leq y$ is at least
$$ \left(\frac{1}{4}+o(1)\right)\frac{x}{2^{\pi(y)}(\log x)^2}.$$
\end{lem}
\begin{proof}
Let $\ell= \pi(y)-1$. Let $p_j$ denote the $j$-th prime, with $p_1=2$. For each $\mathbf{v}=(v_1, \dots, v_{\ell})\in \{-1, 1\}^{\ell}$ we define $S_x(\V)$ to be the set of prime numbers $q\equiv 1 \bmod 8$ with $q\in (x^{1/3}, x^{1/2})$ and such that 
$$ \left(\frac{p_j}{q}\right)=v_j \ \text{ for all } 2\leq j\leq \ell.$$ 
Since $q> p_j$ for all  $j\leq \ell$, then by the prime number theorem in arithmetic progressions we get 
\begin{equation}\label{Mean}
\sum_{\V\in \{-1, 1\}^{\ell}}\left|S_x(\V)\right|= \left(\frac{1}{2}+o(1)\right) \frac{\sqrt{x}}{\log x}.
\end{equation}
We now consider the set of fundamental discriminants 
$$ \mathcal{P}_2(x):= \big\{ D= q_1q_2 \  \text{ such that } q_1<q_2, \text{ and } (q_1, q_2) \in S_x(\V)^2 \text{ for some } \V \in  \{-1, 1\}^{\ell}\big\}.$$
Let $D=q_1q_2\in \mathcal{P}_2(x)$. Then $D\equiv 1\pmod 8$ and hence $\chi_D(2)=\left(\frac{D}{2}\right)=1$. Furthermore, it follows from the law of quadratic reciprocity that for any $2\leq j\leq \ell$ we have
$$ \chi_D(p_j)= \left(\frac{q_1}{p_j}\right)\left(\frac{q_2}{p_j}\right)= \left(\frac{p_j}{q_1}\right)^2=1, $$
since $q_1\equiv q_2 \equiv 1\pmod 4$, and $q_1, q_2\in S_x(\V)$ for some $\V \in  \{-1, 1\}^{\ell}$. Thus, by multiplicativity it follows that for all $D\in \mathcal{P}_2(x)$ we have $\chi_D(n)=1$ for all $n\leq y$.

To complete the proof, we need to establish the desired lower bound on $|\mathcal{P}_2(x)|$. Note that 
\begin{equation}\label{SetFundamental}
\left|\mathcal{P}_2(x)\right|= \sum_{\V\in \{-1, 1\}^{\ell}}\binom{\left|S_x(\V)\right|}{2} = \frac{1}{2} \sum_{\V\in \{-1, 1\}^{\ell}}\left|S_x(\V)\right|^2 +O\left(\frac{\sqrt{x}}{\log x}\right). 
\end{equation}
by \eqref{Mean}. Moreover, by the Cauchy-Schwarz inequality and \eqref{Mean} we obtain
$$
 \sum_{\V\in \{-1, 1\}^{\ell}}\left|S_x(\V)\right|^2 \geq \frac{1}{2^{\ell}} \left(\sum_{\V\in \{-1, 1\}^{\ell}}\left|S_x(\V)\right|\right)^2 \geq \left(\frac{1}{4}+o(1)\right)\frac{x}{2^{\ell} (\log x)^2}. 
$$
Inserting this bound in \eqref{SetFundamental} completes the proof.  
\end{proof}
\begin{proof}[Proof of Theorem \ref{NoZerosThm}]
By Lemma \ref{DiscriminantsChiD1} there are at least $x^{1-1/\log\log x}$ fundamental discriminants $0<D\leq x$ such that $\chi_D(n)=1$ for all $n\leq (\log x)/2 $. The result then follows from Lemma \ref{NoZerosLem}.
\end{proof}

\section*{Acknowledgements}
Y.L. is supported by a junior chair of the Institut Universitaire de France.  O.K. and M.M.  would like to thank the Max Planck Institute for
Mathematics (Bonn) for the hospitality and excellent working conditions during their work on this project. M.M. also acknowledges support by the Austrian Science Fund (FWF), stand-alone project P 33043  ``Character sums, L-functions and applications'' and by the Ministero della Istruzione e della Ricerca ``Young Researchers Program Rita Levi Montalcini''.  O.K. would like to thank Mittag-Leffler Institute for providing excellent working conditions during the final stages of preparation of the present manuscript. The authors are also grateful to the Heilbronn Focused Research grant for support.

\bibliographystyle{plain}
\bibliography{fekete}

\begin{thebibliography}{10}

\bibitem{Angelo-Xu}
R.~Angelo and M.~Xu.
\newblock On a {T}ur\'an conjecture and random multiplicative functions.
\newblock {\em Quart. J. of Math.}, 74(2):767--777, 2023.

\bibitem{Armon}
M.~V. Armon.
\newblock Averages of real character sums.
\newblock {\em J. Number Theory}, 77(2):209--226, 1999.

\bibitem{Aymone}
M.~Aymone, W.~Heap, and J.~Zhao.
\newblock Sign changes of the partial sums of a random multiplicative function.
\newblock {\em Bull. Lond. Math. Soc.}, 55(1):78--89, 2023.

\bibitem{BaMo}
R.~C. Baker and H.~L. Montgomery.
\newblock Oscillations of quadratic {$L$}-functions.
\newblock In {\em Analytic number theory ({A}llerton {P}ark, {IL}, 1989)},
  volume~85 of {\em Progr. Math.}, pages 23--40. Birkh\"{a}user Boston, Boston,
  MA, 1990.

\bibitem{BGHS}
W.~D. Banks, M.~Z. Garaev, D.~R. Heath‐Brown, and I.~E. Shparlinski.
\newblock Density of non‐residues in burgess‐type intervals and
  applications.
\newblock {\em Bull. Lond. Math. Soc.}, 40(1):88--96, 2008.

\bibitem{Bengo}
P.~Bengoechea.
\newblock Galois action on special theta values.
\newblock {\em J. Th\'{e}or. Nombres Bordeaux}, 28(2):347--360, 2016.

\bibitem{BGGK}
J.~Bober, L.~Goldmakher, A.~Granville, and D.~Koukoulopoulos.
\newblock The frequency and the structure of large character sums.
\newblock {\em J. Eur. Math. Soc. (JEMS)}, 20(7):1759--1818, 2018.

\bibitem{BC}
P.~Borwein and K.~S. Choi.
\newblock Explicit merit factor formulae for {F}ekete and {T}uryn polynomials.
\newblock {\em Trans. Amer. Math. Soc.}, 354(1):219--234, 2002.

\bibitem{borwein2001extremal}
P.~Borwein, K.~S. Choi, and S.~Yazdani.
\newblock An extremal property of {F}ekete polynomials.
\newblock {\em Proc. Amer. Soc.}, 129(1):19--27, 2001.

\bibitem{ErdBor}
P.~Borwein, T.~Erd\'{e}lyi, and G.~K\'{o}s.
\newblock Littlewood-type problems on {$[0,1]$}.
\newblock {\em Proc. London Math. Soc. (3)}, 79(1):22--46, 1999.

\bibitem{M-N-T-1}
S.~Chidambaram, J.~, Min\'{a}\v{c}, T.~Nguyen, and T.~Duy.
\newblock {F}ekete polynomials of principal {D}irichlet characters.
\newblock {\em https://arxiv.org/abs/2307.14896}, 2023.

\bibitem{Chow}
S.~Chowla.
\newblock Note on a {D}irichlet {L}- function.
\newblock {\em Acta Arith.}, 1(1):113--114, 1936.

\bibitem{CGPS}
B.~Conrey, A.~Granville, B.~Poonen, and K.~Soundararajan.
\newblock Zeros of {F}ekete polynomials.
\newblock {\em Ann. Inst. Fourier (Grenoble)}, 50(3):865--889, 2000.

\bibitem{ConreySound}
J.~B. Conrey and K.~Soundararajan.
\newblock Real zeros of quadratic {D}irichlet {$L$}-functions.
\newblock {\em Invent. Math.}, 150(1):1--44, 2002.

\bibitem{Dav}
H.~Davenport.
\newblock {\em Multiplicative number theory}, volume~74 of {\em Graduate Texts
  in Mathematics}.
\newblock Springer-Verlag, New York, third edition, 2000.
\newblock Revised and with a preface by Hugh L. Montgomery.

\bibitem{BMT}
R.~de~la Bret\`eche, M.~Munsch, and G.~Tenenbaum.
\newblock Small {G}\'{a}l sums and applications.
\newblock {\em J. Lond. Math. Soc. (2)}, 103(1):336--352, 2021.

\bibitem{Poonen}
A.~Dembo, B.~Poonen, Q.~Shao, and O.~Zeitouni.
\newblock Random polynomials having few or no real zeros.
\newblock {\em J. Amer. Math. Soc.}, 15(4):857--892, 2002.

\bibitem{ErdSurv}
T.~Erd\'{e}lyi.
\newblock Recent progress in the study of polynomials with constrained
  coefficients.
\newblock In {\em Trigonometric sums and their applications}, pages 29--69.
  Springer, Cham, 2020.

\bibitem{EO}
P.~Erd\H{o}s and A.~C. Offord.
\newblock On the number of real roots of a random algebraic equation.
\newblock {\em Proc. London Math. Soc. (3)}, 6:139--160, 1956.

\bibitem{FeketePolya}
M.~Fekete and G.~P{\'o}lya.
\newblock {\"U}ber ein problem von {L}aguerre.
\newblock {\em Rendiconti del Circolo Matematico di Palermo (1884-1940)},
  34(1):89--120, 1912.

\bibitem{G-H}
N.~Geis and G.~Hiary.
\newblock Counting sign changes of partial sums of random multiplicative
  functions.
\newblock {\em https://arxiv.org/abs/2311.16358}, 2023.

\bibitem{GrSo}
A.~Granville and K.~Soundararajan.
\newblock The distribution of values of {$L(1,\chi_d)$}.
\newblock {\em Geom. Funct. Anal.}, 13(5):992--1028, 2003.

\bibitem{G-S-max-character}
A.~Granville and K.~Soundararajan.
\newblock Large character sums: pretentious characters and the
  {P}\'{o}lya-{V}inogradov theorem.
\newblock {\em J. Amer. Math. Soc.}, 20(2):357--384, 2007.

\bibitem{Harper-short}
A.~Harper.
\newblock The typical size of character and zeta sums is $o(\sqrt{x})$.
\newblock {\em https://arxiv.org/abs/2301.04390}, 2023.

\bibitem{HarperMaks}
A.~Harper, A.~Nikeghbali, and M.~Radziwill.
\newblock A note on {H}elson's conjecture on moments of random multiplicative
  functions.
\newblock {\em Analytic Number Theory, Conf. in honor of Helmut Maier's 60th
  birthday}, 2015.

\bibitem{Harperhigh}
A.~J. Harper.
\newblock Moments of random multiplicative functions, {II}: {H}igh moments.
\newblock {\em Algebra Number Theory}, 13(10):2277--2321, 2019.

\bibitem{HBsieve}
D.~R. Heath-Brown.
\newblock A mean value estimate for real character sums.
\newblock {\em Acta Arith.}, 72:235--275, 1995.

\bibitem{Heilbr}
H.~Heilbronn.
\newblock On real characters.
\newblock {\em Acta Arith.}, 2:212--213, 1936.

\bibitem{Hoholdtmerit}
T.~Hoholdt and H.~E. Jensen.
\newblock Determination of the merit factor of {L}egendre sequences.
\newblock {\em IEEE Transactions on Information Theory}, 34(1):161--164, 1988.

\bibitem{AyeshaLamzouri}
A.~Hussain and Y.~Lamzouri.
\newblock The limiting distribution of legendre paths.
\newblock {\em Preprint, 20 pages. https://arxiv.org/abs/2304.13025}, 2023.

\bibitem{JKS}
J.~Jedwab, D.~J. Katz, and K-U. Schmidt.
\newblock Advances in the merit factor problem for binary sequences.
\newblock {\em Journal of Combinatorial Theory, Series A}, 120(4):882--906,
  2013.

\bibitem{jedwab2013littlewood}
J.~Jedwab, D.~J Katz, and K-U. Schmidt.
\newblock Littlewood polynomials with small ${L}^4$ norm.
\newblock {\em Advances in Mathematics}, 241:127--136, 2013.

\bibitem{JHH}
J.~M. Jensen, H.~E. Jensen, and T.~Hoholdt.
\newblock The merit factor of binary sequences related to difference sets.
\newblock {\em IEEE Transactions on Information Theory}, 37(3):617--626, 1991.

\bibitem{Jung}
J.~Jung.
\newblock On the sparsity of positive-definite automorphic forms within a
  family.
\newblock {\em J. Anal. Math.}, 129:105--138, 2016.
\newblock With an appendix by Jung and Sug Woo Shin.

\bibitem{Jut1}
M.~Jutila.
\newblock On character sums and class numbers.
\newblock {\em J. Number Theory}, 5:203--214, 1973.

\bibitem{Kalmynin}
A.~Kalmynin.
\newblock Quadratic characters with positive partial sums.
\newblock {\em Mathematika}, 69(1):90--99, 2023.

\bibitem{Kerr-Klurman}
B.~Kerr and O.~Klurman.
\newblock How negative can $\sum_{n\le x}\frac{f(n)}{n}$ be?
\newblock {\em https://arxiv.org/pdf/2203.09448.pdf}, 2022.

\bibitem{K-L-M}
O.~Klurman, Y.~Lamzouri, and M.~Munsch.
\newblock ${L}_q$ norms and {M}ahler measure of {F}ekete polynomials.
\newblock {\em https://arxiv.org/abs/2306.07156}, 2023.

\bibitem{Koukoul}
D.~Koukoulopoulos.
\newblock {\em The distribution of prime numbers}, volume 203 of {\em Graduate
  Studies in Mathematics}.
\newblock American Mathematical Society, Providence, RI, [2019] \copyright
  2019.

\bibitem{Lamzouri}
Y.~Lamzouri.
\newblock The distribution of large quadratic character sums and applications.
\newblock {\em To appear in Algebra Number Theory}, 2022.

\bibitem{LLR2}
Y.~Lamzouri, S.~Lester, and M.~Radziwi\l~\l.
\newblock Discrepancy bounds for the distribution of the {R}iemann
  zeta-function and applications.
\newblock {\em J. Anal. Math.}, 139(2):453--494, 2019.

\bibitem{LLR1}
Y.~Lamzouri, S.~Lester, and M.~Radziwill.
\newblock An effective universality theorem for the {R}iemann zeta function.
\newblock {\em Comment. Math. Helv.}, 93(4):709--736, 2018.

\bibitem{littlewood1968some}
J.~E. Littlewood.
\newblock {\em Some problems in real and complex analysis}.
\newblock DC Heath, 1968.

\bibitem{LO1}
J.~E. Littlewood and A.~C. Offord.
\newblock On the number of real roots of a random algebraic equation. {II}.
\newblock {\em Math. Proc. Cambridge Philos. Soc.}, 35(2):133--148, 1943.

\bibitem{LO2}
J.~E. Littlewood and A.~C. Offord.
\newblock On the number of real roots of a random algebraic equation. {III}.
\newblock {\em Rec. Math. [Mat. Sbornik] N.S.}, 12(54):277--286, 1943.

\bibitem{Lounumerics}
S.~R. Louboutin.
\newblock On {C}howla's hypothesis implying that {$L(s,\chi )\>0$} for {$s\>0$}
  for real characters {$\chi$}.
\newblock {\em Bull. Korean Math. Soc.}, 60(1):1--22, 2023.

\bibitem{Debrecen}
S.~R. Louboutin and M.~Munsch.
\newblock On positive real zeros of theta and {$L$}-functions associated with
  real, even and primitive characters.
\newblock {\em Publ. Math. Debrecen}, 83(4):643--665, 2013.

\bibitem{jnttheta}
S.~R. Louboutin and M.~Munsch.
\newblock The second and fourth moments of theta functions at their central
  point.
\newblock {\em J. Number Theory}, 133(4):1186--1193, 2013.

\bibitem{M-N-T}
J.~Min\'{a}\v{c}, T.~Nguyen, and T.~Duy.
\newblock On the arithmetic of generalized {F}ekete polynomials.
\newblock {\em To appear in Experimental {M}athematics}, 2022.

\bibitem{MR4490459}
J.~Min\'{a}\v{c}, T.~Nguyen, and T.~Duy.
\newblock Fekete polynomials, quadratic residues, and arithmetic.
\newblock {\em J. Number Theory}, 242:532--575, 2023.

\bibitem{Mont-large}
H.~L. Montgomery.
\newblock An exponential polynomial formed with the {L}egendre symbol.
\newblock {\em Acta Arith.}, 37:375--380, 1980.

\bibitem{Pol}
G.~P\'{o}lya.
\newblock Verschiedene {B}emerkung zur {Z}ahlentheorie.
\newblock {\em Jber. Deutsch. Math. Verein}, 28(3-4):31--40, 1919.

\bibitem{SarBach}
P.~Sarnak.
\newblock Letter to {B}achmat on positive definite {L}-functions.
\newblock {\em https://publications.ias.edu/sites/default/files/positive},
  2011.

\bibitem{Vaaler}
J.~D. Vaaler.
\newblock Some extremal functions in {F}ourier analysis.
\newblock {\em Bull. Amer. Math. Soc. (N.S.)}, 12(2):183--216, 1985.

\end{thebibliography}

\end{document}